\documentclass[a4paper,11pt,twoside,reqno]{article}

\usepackage{hyperref}

\usepackage[english]{babel}
\usepackage{fancyhdr}
\usepackage{fullpage} 
\usepackage[dvips]{graphicx}
\usepackage[T1]{fontenc}
\usepackage{fouriernc}
\usepackage[latin1,applemac]{inputenc}
\usepackage{amsmath}
\usepackage{amsfonts}
\usepackage{amssymb}
\usepackage{mathrsfs}
\usepackage{amsthm}
\usepackage{latexsym}
\usepackage{array} 
\usepackage{mathrsfs}
\usepackage{amsxtra} 
\usepackage{amscd} 
\usepackage{mathtools} 
\usepackage{verbatim}

\usepackage{enumerate}
\usepackage{amsbsy}

\usepackage[toc,page]{appendix}

\usepackage[usenames]{color}
\definecolor{beige}{rgb}{0.96, 0.96, 0.86}
\definecolor{airforceblue}{rgb}{0.36, 0.54, 0.66}
\definecolor{antiquefuchsia}{rgb}{0.57, 0.36, 0.51}
\definecolor{awesome}{rgb}{1.0, 0.13, 0.32}

\newtheorem{theorem}{Theorem}[section]
\newtheorem{definition}[theorem]{Definition}
\newtheorem{proposition}[theorem]{Proposition}
\newtheorem{lemma}[theorem]{Lemma}
\newtheorem{corollary}[theorem]{Corollary}

\newtheorem{thm}[theorem]{Theorem}
\newtheorem{lem}[theorem]{Lemma}

\newtheorem{prop}[theorem]{Proposition}

\theoremstyle{definition}
\newtheorem{remark}[theorem]{Remark}
\newtheorem{example}[theorem]{Example}
\newtheorem{rem}[theorem]{Remark}

\newtheorem{assum}[theorem]{Assumption}

\newtheorem{assumptions}[theorem]{Assumptions}

\numberwithin{equation}{section}

\renewcommand{\epsilon}{\varepsilon}
\renewcommand{\tilde}{\widetilde}

\newcommand{\T}{\mathsf{T}}
\newcommand{\A}{\operatorname{\mathsf{A}}}
\newcommand{\B}{\operatorname{\mathsf{B}}}
\newcommand{\Q}{\operatorname{\mathsf{Q}}}
\newcommand{\M}{\operatorname{\mathsf{M}}}
\newcommand{\X}{\operatorname{\mathsf{X}}}
\newcommand{\Y}{\operatorname{\mathsf{Y}}}
\newcommand{\Z}{\operatorname{\mathsf{Z}}}
\newcommand{\Proc}{\operatorname{\mathsf{P}}}
\newcommand{\sigmab}{\boldsymbol{\sigma}}

\usepackage{xcolor}
\definecolor{rose}{rgb}{1.0,0.33,0.64}
\hypersetup{
    colorlinks,
    linkcolor={green!60!blue},
    citecolor={green!60!blue},
    urlcolor={green!60!blue}
}

\newcommand{\we}{\wedge}
\newcommand{\ol}{\overline}
\newcommand{\ul}{\underline}
\newcommand{\ola}{\overleftarrow}

\newcommand{\ba}{\begin{array}}
\newcommand{\ea}{\end{array}}
\newcommand{\be}{\begin{equation}}
\newcommand{\ee}{\end{equation}}

\def\dbE{\mathbb{E}}
\def\dbF{\mathbb{F}}
\def\dbG{\mathbb{G}}

\def\dbN{\mathbb{N}}
\def\dbP{\mathbb{P}}
\def\dbR{\mathbb{R}}
\def\dbS{\mathbb{S}}

\def\a{\alpha}
\def\b{\beta}
\def\g{\gamma}
\def\d{\delta}
\def\e{\varepsilon}

\def\n{\nu}

\def\t{\tau}
\def\f{\varphi}
\def\th{\theta}
\def\o{\omega}

\def\Th{\Theta}

\def\O{\Omega}
\def\cD{{\cal D}}
\def\cE{{\cal E}}

\def\cG{{\cal G}}

\def\cJ{{\cal J}}

\def\cP{{\cal P}}

\def\cT{{\cal T}}

\def\ch{\textsc{h}}

\def\no{\noindent}

\def\ms{\medskip}

\def\q{\quad}

\def\pa{\partial}
\def\cd{\cdot}

\def\1{{\bf 1}}

\begin{document}


\title{Viscosity solutions of path-dependent PDEs with randomized time}

\date{}

\author{Zhenjie  {\sc Ren}
\footnote{Universit\'e Paris-Dauphine, PSL Research University, CNRS, UMR [7534], Ceremade, 75016 Paris, France, ren@ceremade.dauphine.fr.}
\and
Mauro {\sc Rosestolato}
\footnote{CMAP, Ecole Polytechnique Paris, mauro.rosestolato@polytechnique.edu. Research supported by ERC 321111 Rofirm.
\newline
The authors are sincerely grateful to Nizar Touzi for valuable discussions.
}
}

\maketitle



\begin{abstract}
We introduce a new definition of viscosity solution to path-dependent partial differential equations, which is a slight modification of the definition introduced in \cite{EKTZ}. With the new definition, we prove the two important results till now missing in the literature, namely, a general stability result and a comparison result for semicontinuous sub-/super-solutions. As an application, we prove the existence of viscosity solutions using the Perron method. Moreover,
 we connect viscosity solutions of path-dependent PDEs with
viscosity solutions of partial differential equations on Hilbert spaces.

\end{abstract}

\section{Introduction}

This paper 
studies 
 viscosity solutions 
of
 the fully nonlinear path-dependent partial differential equation
\begin{equation}
  \label{PPDEintro}
-\pa_t u(t,\o)-G\big(t,\o, u(t,\o), \pa_\o u(t,\o), \pa^2_{\o\o} u(t,\o)\big) = 0\q\mbox{on}\q [0,T)\times \O.
\end{equation}
Here, $T>0$ is a given terminal time and $\o\in\O$ is a continuous path from $[0,T]$ to $\dbR^m$ starting from the origin. The path derivatives $\pa_t, \pa_\o, \pa^2_{\o\o}$ were first introduced in the work of Dupire \cite{Dupire}. See also \cite{CF} for the related It\^o calculus.  Such equations arise naturally in many applications. For example, the dynamic programming equation associated
with a stochastic control problem of
 non-Markov diffusions (see \cite{ETZ1}) and the one associated 
with a stochastic differential game with non-Markov dynamics (see \cite{PZ}) both fall in the class of equation \eqref{PPDEintro}. The notion of nonlinear path-dependent partial differential equations was first proposed by Peng \cite{Peng}. We also refer to Peng and Wang \cite{PW} for a study on classical solutions of semilinear equations.

The notion of viscosity solutions studied in this paper is a slight modification over the one introduced in Ekren et al.\ (\cite{EKTZ}) in the semilinear context and further extended to the fully nonlinear case in \cite{ETZ1,ETZ2}. Following the lines of the classical Crandall and Lions notion of viscosity solutions  (\cite{CL}), supersolutions and subsolutions are defined through tangent test functions. However, while Crandall and Lions consider pointwise tangent functions, the tangency conditions in the path-dependent setting is in the sense of the expectation with respect to an appropriate class of probability measures $\cP$. We refer to \cite{Ren2014} for an overview, and to \cite{CFGRT,Ekren, AA, Keller,  Ren-ell, RT, Ren, Ren2015} for some of the generalizations. 

Regardless of the successful development mentioned above, there 
were some 
 difficulties for the theory of  viscosity solutions 
coming from 
 the definition adopted in \cite{EKTZ}. First, a good stability result was missing. In \cite{ETZ1} the authors proved a stability result in the sense that if a sequence of uniformly continuous solutions $u^n$ uniformly converges to a uniformly continuous function $u$, then $u$ is also a solution. The assumptions of the uniform continuity and the uniform convergence are often too strong for applications, for example, for the Perron method to prove the existence of viscosity solution.
Secondly,
also a
comparison result for semicontinuous viscosity sub-/super-solutions
was missing. 
The aim of the present paper is to fill these two theoretical gaps. We do this by slightly modifying the definition in \cite{EKTZ}. This modification is sufficient to let us overcome the technical difficulties for proving stability and comparison results, but does not compromise the other results till now obtained in the literature, such as existence.

 In the previous definition, 
adopted in 
\cite{EKTZ},
the test function $\f$ is urged to be tangent to the (sub)solution $u$ at a point (say $0$) in the sense that
\begin{equation}
  \label{intro-os}
(u-\f)(0) =\max_{\t\in\cT} \sup_{\dbP\in \cP}\dbE^\dbP\big[(u-\f)(\t, \B)\big],
\end{equation}
where $\cT$ is the set of all stopping times taking values in $[0,T]$, and $\cP$ is a family of probability measures on the path space $\O$ on which $\B$ is the canonical process. In the arguments for proving the stability and the comparison, we often need to solve the optimal stopping problem on the right hand side of \eqref{intro-os}, which is not a simple task and requires
the uniform continuity of $u$ (see \cite{ETZ-os}). This turns out to be one of the main difficulties in improving the stability and the comparison results. In the present paper, in order to overcome this fundamental difficulty, we randomize the optimal stopping problem, and the new definition reads
\begin{equation*}
  (u-\f)(0) = \sup_{\dbP\in \tilde\cP}\dbE^\dbP\big[(u-\f)(\T, \B)\big],
\end{equation*}
where $\tilde\cP$ is a family of probability measures on the time-path product space $[0,T]\times \O$ on which $(\T,\B)$ is the canonical process. It turns out that the randomized problem can be solved for less regular functions $u$, and is more stable. With this change of definition, we manage to prove a general stability result for semicontinuous viscosity solutions
(Theorem~\ref{2018-05-17:00}). The recipe of our proof is composed of the classical argument for stability in \cite{user} and the measurable selection theorem. 
Moreover, by using
the stability result, we are also able
 to prove the existence of viscosity solution through Perron's method
(Theorem~\ref{thm: Perron}).
Further, 
we prove a comparison result for semicontinuous sub-/super-solutions under some strong assumptions
(Theorem~\ref{2018-06-18:02}),
by approximating semicontinuous solutions with 
Lipschitz continuous 
solutions to approximating equations.

Another major contribution of this paper is to connect the path-dependent partial differential equation \eqref{PPDEintro} with the partial differential equation
\begin{equation}
  \label{introHilbert}
  -u_t-\langle Ax,D_x u\rangle - G(t,x,u,D_{x_0}u,D^2_{x_0x_0}u)=0\quad \mbox{on }\quad(0,T)\times H,
\end{equation}
where $H$ is a Hilbert space into which the path space $\O$ can embed, $A$ is an unbounded operator, and
$D_{x_0},D_{x_0x_0}^2$ are the first- and second-order differentials with respect to a finite dimensional subspace of $H$.
Both equations
\eqref{PPDEintro}
and
\eqref{introHilbert}
 can be used to characterize the value function of non-Markov stochastic control problem. In this paper, we prove that a viscosity solution to the path-dependent equation
\eqref{PPDEintro}, under some regularity assumption, is also a viscosity solution to the corresponding equation
\eqref{introHilbert}
(Theorem~\ref{2017-01-31:05}).
The theory of viscosity solutions for partial differential equations on Hilbert spaces
 (we mainly refer to \cite{Fabbri})
is designed of a large class of equations not rescrited to those of delay type, as
\eqref{introHilbert}.
We notice that, 
till now,
when applied to PDEs 
of the form
\eqref{introHilbert},
such a theory can deliver a comparison result under more general assumption on the nonlinearity function $G$, but only for more regular sub-/super-solutions. 
Moreover,
 the theory of  path-dependent equations
as here developed
 can treat solutions (semi)continuous in
the $L^\infty$-norm, which 
cannot be settled in a 
 Hilbert space framework.

\smallskip
The rest of the paper is organized as follows. Section \ref{sec:notation} introduces the main notations. Section~\ref{sec:def} presents the modified definition of the viscosity solutions to the path-dependent partial differential equations. In Section \ref{sec:stability} we prove the  stability result 
(Theorem~\ref{2018-05-17:00}), and using it
 in Section \ref{sec:perron} we prove the existence of viscosity solution with Perron's method
(Theorem~\ref{thm: Perron}). In Section \ref{sec:comparison} we show the comparison result for semicontinuous solutions
(Theorem~\ref{2018-06-18:02}). In Section \ref{sec:hilbert} we clarify the connection between the path-dependent equation and the
correponding  equation on the Hilbert space
(Theorem~\ref{2017-01-31:05}). Finally, we complete some proofs in Appendix.

\section{Notations}\label{sec:notation}

\paragraph{Canonical Space.}
Let $m>0$ be a natural number, $T>0$ be a real number.
Define
$$
\Omega\coloneqq \{\o\in C([0,T],\mathbb{R}^m)\colon \o(0)=0\}.
$$
We denote by
$|\cdot|$ the Euclidean norm in $\mathbb{R}^m$ and by
 $|\cdot|_\infty$ the uniform norm on $\Omega$.
In this paper, we study equations set on the spacetime space:
\begin{equation*}
\Theta\coloneqq [0,T]\times\Omega.
\end{equation*}
For technical reasons, we also often work on the enlarged canonical space:
\begin{equation*}
\tilde \Theta\coloneqq 
\Theta\times \Omega\times \Omega\times 
C_0([0,T],\mathbb{S}^m)
\end{equation*}
where $\mathbb{S}^m$
is the space of symmetric $m\times m$ real matrices 
endowed with the supremum norm
and
$$
C_0([0,T],\mathbb{S}^m)\coloneqq \{f\in C([0,T],\mathbb{S}^m)\colon f(0)=0\}.
$$

\no Hereafter, we will reserve the
letter  $\omega$ for a generic element of $\Omega$,
the letter $\theta$ for the generic couple 
$\theta=(t,\omega)\in \Theta$,
and the letter $\vartheta$ for the generic element
$\vartheta=(\theta,a,\mu,q)\in \tilde \Theta$.
We introduce on $\Theta$ the pseudo-metric $d_\infty$ defined by
\begin{equation*}
  d _\infty(\theta,\theta')\coloneqq |t-t'|+|\omega_{t\wedge \cdot}-\omega_{t'\wedge \cdot}|_\infty.
\end{equation*}

\begin{rem}
In this paper, without being otherwise stated, the (semi-)continuity on the canonical spaces is under $|\cd|_\infty$.  The (semi-)continuity related to other pseudo-metrics will be explicitly expressed.    
\end{rem}

\begin{rem}
Note the following facts:
\begin{itemize}
\item Let $\mathbb{F}\coloneqq\{\mathcal{F}_t\}_{t\in[0,T]}$ denote the canonical filtration on $\Omega$. All $d_\infty$-continuous functions are $\dbF$-progressively measurable.
\item $\lim_{n\rightarrow\infty } |\th^n -\th|_\infty =0$ implies that $\lim_{n\rightarrow \infty} d_\infty(\th^n,\th)=0$, and thus all $d_\infty$-continuous functions are continuous.
\end{itemize}
\end{rem}

Let 
$\mathbb{F}^{\mathbb{S}}\coloneqq\{\mathcal{F}^{\mathbb{S}}_t\}_{t\in[0,T]}$ denote the canonical filtration on $C_0([0,T],\mathbb{S}^m)$.
We introduce on $\tilde \Theta$ the filtration 
$\mathbb{G}\coloneqq \{\mathcal{G}_t\}_{t\in[0,T]}
$ defined by
\begin{equation*}
\mathcal{G}_t\coloneqq  
 \left( 
   \sigma\left\{
     [0,r],\ r\leq t
   \right\}  \otimes \mathcal{F}_t  \right) 
 \otimes \mathcal{F}_t \otimes \mathcal{F}_t\otimes \mathcal{F}_t^{\mathbb{S}}
\qquad \forall t\in [0,T].
\end{equation*}

\paragraph{Shifted functions}.
For $\omega,\omega'\in \Omega$, $t\in[0,T]$, we define
\begin{equation*}
  (\omega  \otimes_t \omega')_s\coloneqq
  \begin{dcases}
    \omega_s & \mbox{if }s\in[0,t]\\
    \omega_t+\omega'_{s-t}    & \mbox{if }s\in(t,T].
  \end{dcases}
\end{equation*}
For any function $\xi\colon \Omega\rightarrow R$ taking values in some set $R$, and for any $\theta=(t,\omega)\in \Theta$, we denote
\begin{equation*}
  \xi^\theta(\omega')\coloneqq \xi(\omega \otimes _t  \omega')\qquad \forall \omega'\in \Omega.
\end{equation*}
Similarly, given a function $u\colon \Theta\rightarrow R$, we denote, for $\theta\in \Theta$,
\begin{equation*}
  u^\theta(t',\omega')\coloneqq 
u\big((t+t')\wedge T,
\omega \otimes_t \omega'\big)\qquad \forall 
\theta'\in \Theta.
\end{equation*}

\noindent Clearly, if $\xi$ is $\mathcal{F}_T$-measurable then $\xi^\theta$ is $\mathcal{F}_{T-t}$-measurable, and if $X$ $\mathbb{F}$-adapted then so is $X^\theta$.

Similarly, we can also shift functions defined on the enlarged canonical space $\tilde \Th$. For $t'\in[0,T]$,
$\theta\in \Theta$, $\vartheta\in \tilde \Theta$, we denote
\begin{equation*}
\theta_{t'\wedge \cdot}\coloneqq
(t'\wedge t,\omega_{t' \wedge\cdot}),
\qquad
\vartheta_{t'\wedge \cdot}= ( \theta_{t'\wedge \cdot},a_{t'\wedge \cdot},\mu_{t'\wedge \cdot},q_{t'\wedge \cdot}).
\end{equation*}
Given a function $v\colon \tilde\Theta\rightarrow R$ , we define for $\vartheta\in \tilde\Theta$ and $s\in [0,T]$
\begin{equation*}
v^{s,\vartheta}(\vartheta') \coloneqq  v^{\vartheta_{s\we\cd}}(\vartheta') 
= v\Big(\big(t\we s+t'\big)\we T, \o\otimes_{t\we s} \o', a\otimes_{t\we s} a', \mu\otimes_{t\we s}\mu', q\otimes_{t\we s} q'\Big)\qquad \forall 
\vartheta'\in \tilde\Theta.
\end{equation*}
Further, for a $\dbG$-stopping time $\t$, we define
\begin{equation*}
v^{\t, \vartheta} \coloneqq  v^{\t(\vartheta),\vartheta}.
\end{equation*}
In particular, note that $v^{\t,\vartheta} = v^{t\we\t,\vartheta}$.

\paragraph{Probability Space.}
We denote by $\mathcal{ P}$ the set of probability measures on $(\tilde \Theta,\mathcal{G}_T)$.
Unless otherwise specified, the set $\mathcal{ P}$ is always  endowed with the topology of the weak convergence.
We recall that
 $\tilde \Theta$ is a Polish space with respect to the product topology, hence $\mathcal{ P}$ is a Polish space too.
We denote by $\T,\B,\A,\M,\Q$ the first, second, third, fourth, and fifth projection of $\tilde \Theta$,
respectively
\begin{align*}
    \T\colon \tilde \Theta\rightarrow \mathbb{R},\ \vartheta \mapsto t\qquad
\ \B\colon \tilde \Theta\rightarrow \Omega,
\ \vartheta \mapsto \omega\qquad
\  \A\colon \tilde \Theta\rightarrow \Omega,
\ \vartheta \mapsto a\\
\  \M\colon \tilde \Theta\rightarrow \Omega,
\ \vartheta \mapsto \mu\qquad
\ \Q\colon \tilde \Theta\rightarrow C_0([0,T],\mathbb{S}^m),\ \vartheta \mapsto q.
\end{align*}
We stress the fact that $\T$ is a $\mathbb{G}$-stopping time and  $\B,\A,\M,\Q$ are all $\mathbb{G}$-adapted processes. In order to simplify the notations, we also denote the quintuple
$(\T,\B,\A,\M,\Q)$ by $\X$,
$$\X\coloneqq (\T,\B,\A,\M,\Q).$$

\begin{rem}\label{rem:2shifts}
Using the notations defined previously, we note that 
\begin{equation*}
\T^{\t,\vartheta}(\vartheta') = \big(t\we\t(\vartheta) +t'\big)\we T,\q
\B^{\t,\vartheta}(\vartheta') = \o\otimes_{t\we\t(\vartheta)}\o', \q 
\forall \vartheta,\vartheta'\in \tilde\Th,
\end{equation*}
where $\t$ is a $\dbG$-stopping time. Further, given a function $u\colon\Th\rightarrow R$, we have
\begin{equation*}
u\big(\T^{\t,\vartheta}(\vartheta'),\B^{\t,\vartheta}(\vartheta')\big)
= u\Big( \big(t\we\t(\vartheta) +t'\big)\we T, \o\otimes_{t\we\t(\vartheta)}\o' \Big) = u^{t\we\t(\vartheta),\o}(\th').
\end{equation*}
\end{rem}

In this paper, we will use the following probability family on the enlarged canonical space $\tilde \Th$ to define the viscosity solutions to path-dependent PDEs.

\begin{definition}
  For  $L>0$, 
  we define
  the subset of $\tilde \Theta$
  \begin{equation}\label{2016-09-27:01}
\begin{multlined}[c][0.92\displaywidth]          \tilde \Theta_{L}\coloneqq
 \left\{ \vartheta\in \tilde \Theta\colon  t\in [0,T],\ \omega=a+\mu,
\ a, q \in AC([0,T],\mathbb{R}^m),
\
 | \dot a|_\infty\le L
,\ |  \dot q|_\infty \le L
 \right\},
\end{multlined}
\end{equation}
and the set of probabilities
\begin{equation*}
  \mathcal{ P}_{L}\coloneqq 
 \left\{ \mathbb{P}
\in \mathcal{ P}
\colon \mathbb{P}(\tilde \Theta_{L})=1,\ \M\mbox{ is a
square-integrable }
\mathbb{P}\mbox{-martingale},
\ \langle \M\rangle=\Q\ \mathbb{P}\mbox{-a.s.}
 \right\}.
\end{equation*}
\end{definition}

\no Using the canonical processes, we have for all 
 $\mathbb{P}\in \mathcal{ P}_{L}$ that
\begin{gather}
\T \in [0,T],\
    \B=\A+\M\     \mathbb{P}\mbox{-a.s.}\label{2016-09-28:00}\\
\M\mbox{ is a $\mathbb{P}$-martingale}
\label{2016-09-28:02}\\
\langle \M\rangle=\Q
\
      \mathbb{P}\mbox{-a.s.}
\label{2016-09-29:01}\\
\A\in AC([0,T],\mathbb{R}^m)
\mbox{ and }
|\dot\A|_\infty \le  L\ 
\mathbb{P}\mbox{-a.s.}\label{2016-09-28:03}\\
\Q\in AC([0,T],\mathbb{R}^{m\times m})\mbox{ and }
|\dot\Q|_\infty
\leq L\ 
\mathbb{P}\mbox{-a.s.}
\label{2016-09-28:04}
\end{gather}

\no Recall that $\T$ is a $\dbG$-stopping time, so $\B_{\T\wedge \cdot},\A_{\T\wedge \cdot},\M_{\T\wedge \cdot}$ are $\mathbb{G}$-adapted and $\M_{\T\wedge \cdot}$ is a $\mathbb{P}$-martingale.

We introduce the sublinear and superlinear expectation operators associated with $\mathcal{ P}_{L}$:
\begin{equation}
  \label{eq:2017-02-23:02}
  \overline{\mathcal{E}}_{L}\coloneqq \sup_{\mathbb{P}\in \mathcal{P}_{L}}\dbE^\dbP
\qquad\qquad
  \underline{\mathcal{E}}_{L}\coloneqq \inf_{\mathbb{P}\in \mathcal{P}_{L}}\dbE^\dbP.
\end{equation}


\section{Definition of $\cP_L$-viscosity solution}\label{sec:def}

In this paper, we consider the fully nonlinear parabolic path-dependent PDE (PPDE):
\begin{equation}
  \label{eq:PPDEgeneral}
  - \partial_tu-G(\theta,u, \partial _\omega u,\partial^2 _{\omega\omega} u)=0.
\end{equation}
In \cite{Ren2014,Ren}, it is showed that one can define viscosity solutions for PPDEs via jets.
In this manuscript, we start directly from the definition via jets.
For $\alpha\in \mathbb{R},\ \beta\in \mathbb{R}^m,\ \gamma\in \mathbb{S}^m$, let
\begin{equation}
  \label{eq:2017-02-23:01}
  \varphi^{\alpha,\beta,\gamma}(\theta)\coloneqq \alpha t+\langle \beta, \omega_t\rangle+\frac{1}{2} \langle \gamma \omega_t,\omega_t\rangle\qquad \forall \theta\in \Theta,
\end{equation}
where $\langle\cdot,\cdot\rangle$ denotes the standard scalar product on $\mathbb{R}^m$.

%

A further tool for the type of localization that we will implement in the definition of viscosity solution if the function $\ch_\delta$.
For any $\delta\in (0,T]$, define the function $\ch_\delta\colon \tilde{ \Theta}\rightarrow [0,T]$ by
\begin{equation}\label{2018-01-30:05}
  \ch_\delta(\vartheta)\coloneqq
  \inf\{s>0\colon d_\infty\big((s,\o),0\big)=s+|\omega_{s\wedge \cdot}|_\infty\geq \delta\}.
\end{equation}
It is not difficult to show that $\ch_\delta$ is continuous.

Let $\theta\in \Theta$, with $t<T$.
For
$u\colon \Theta\rightarrow \mathbb{R}$ upper semicontinuous,
locally bounded from above, 
the subjet 
of $u$ in $\theta$ is  defined by
\begin{equation*}
  \underline{\mathcal{J }}_L u(\vartheta)\coloneqq
  \left\{ 
    (\alpha,\beta,\g)\in \mathbb{R}\times \mathbb{R}^m\times \mathbb{S}^m\colon
    u(\theta)=
    \overline{\mathcal{E}}_{L}
     \left[ 
       (u^\theta-\varphi^{\alpha,\beta,\g})(\T
       \wedge 
       \ch_\delta
       ,\B)
     \right] \ \mbox{for some $\delta\in(0,T-t]$}
  \right\} ,
\end{equation*}
In a symmetric way, for a lower semicontinuous function
$u\colon \Theta\rightarrow \mathbb{R}$,
locally bounded from below,
the superjet
of $u$ in $\theta$ is  defined by
\begin{equation*}
  \overline{\mathcal{J }}_L u(\theta)\coloneqq
  \left\{ 
    (\alpha,\beta,\g)\in \mathbb{R}\times \mathbb{R}^m\times \mathbb{S}^m\colon
    u(\theta)=
    \underline{\mathcal{E}}_{L}
     \left[ 
       (u^\theta-\varphi^{\alpha,\beta,\g})(\T\wedge 
       \ch_\delta
       ,\B)
     \right]\ \mbox{for some $\delta\in(0,T-t]$}
  \right\} .
\end{equation*}


\noindent We denote $\mathbb{\overline R}\coloneqq \mathbb{R}\cup\{-\infty,+\infty\}$.

\begin{definition}[$\mathcal{P}_L$-viscosity sub-/supersolution]\label{def:viscoppde}
Let $G\colon \Theta\times \mathbb{R}\times \mathbb{R}^m\times \mathbb{S}^m\rightarrow \mathbb{\overline R}$ be a function.
  A
$d_\infty$-upper semicontinuous
(\emph{resp.}\ 
$d_\infty$-lower semicontinuous)
 function $u\colon \Theta
\rightarrow \mathbb{R}$,
locally
 bounded from above 
(\emph{resp.}\ locally bounded from below)
 is a
$\mathcal{P}_L$-viscosity subsolution 
(\emph{resp.}\
$\mathcal{P}_L$-viscosity supersolution) 
of 
\eqref{eq:PPDEgeneral} if
\begin{equation*}
  -\alpha-G(\theta,u(\theta),\beta,\gamma)\leq 0\qquad \forall \theta\in \Theta,\ (\alpha,\beta,\gamma)\in \underline{\mathcal{J}}_L u(\theta),
\end{equation*}
\begin{equation*}
\mbox{\emph{(resp.\ }}  -\alpha-G(\theta,u(\theta),\beta,\gamma)\geq 0\qquad \forall \theta\in \Theta,\ (\alpha,\beta,\gamma)\in \overline{\mathcal{J}}_L u(\theta)
\mbox{\emph{).}}
\end{equation*}
A
locally
 bounded continuous function $u$
 is a $\mathcal{P}_L$-viscosity solution
of 
\eqref{eq:PPDEgeneral} 
 if it is both a $\mathcal{P}_L$-viscosity sub- and supersolution of~\eqref{eq:PPDEgeneral}.
\end{definition}

\begin{remark}\label{rem:comparedef}
Let us recall the previous definition of viscosity solution of path-dependent PDEs. Define the probability family $\cP'_L$ on the space $\tilde\Th' \coloneqq \O\times\O\times\O\times C_0\big([0,T],\dbS^m \big)$ (we still use the notations of canonical processes $\B,\A,\M,\Q$):
\begin{multline*}
  \cP'_L\coloneqq \Big\{\dbP\colon 
 \B=\A+\M, \ \M~\mbox{is a $\dbP$-martingale}, \q\langle\M\rangle =\Q,\ \A\in AC([0,T],\dbR^m), \\
 \Q\in AC([0,T],\dbR^{m,m}),\ |\dot A|_\infty\le L, \ |\dot Q|_\infty \le L, \ \dbP\mbox{-a.s.} \Big\}
\end{multline*}
Note that the space $\tilde\Th'$ misses the time dimension, compared to the canonical space $\tilde\Th$ in the present paper. Define the  nonlinear expectation $\ol\cE'_L \coloneqq  \sup_{\dbP\in \cP'_L} \dbE^\dbP $ and the subjet
\begin{equation*}
  \underline{\mathcal{J }}'_L u(\theta)\coloneqq
  \left\{ 
    (\alpha,\beta,\theta)\in \mathbb{R}\times \mathbb{R}^m\times \mathbb{S}^m\colon
    u(\theta)=
    \max_{\t\in \cT}
    \overline{\mathcal{E}}'_L
     \left[ 
       (u^\theta-\varphi^{\alpha,\beta,\theta})(\tau
       \wedge 
       \ch_\delta
       ,\B)
     \right]\ \mbox{for some $\delta\in(0,T-t]$}
  \right\},
\end{equation*}
where $\cT$ is the set of all $\dbF^{\B}$-stopping times.
Similarly we can define the superjet $\ol\cJ'_L$. Then we define $\cP'_L$-viscosity solutions as in Definition \ref{def:viscoppde}, by replacing the jets $\ol\cJ_L, \ul\cJ_L$ by $\ol\cJ'_L, \ul\cJ'_L$.
As we see, in our new definition of subjet
 the function $(u-\f)^\th_{\ch_\d\we\cd}$ reaches its maximum at $0$
 in the sense of $\ol\cE_{L}$ instead of $\max_{\t\in\cT}\ol\cE'_L$. 
By doing so, the maximization over the stopping times $\t\in \cT$ is replaced by the maximization of the laws applied on the canonical variable $\T$.  
Indeed, the new nonlinear expectation $\ol\cE_L$ is a randomized optimal stopping operator. In general, one need fewer assumptions to ensure the existence of
\begin{equation}
  \label{differentoptimization}
\dbP^*\in \arg\max_{\dbP\in \cP_{L}} \dbE^\dbP[f] \q
\mbox{than that of}\q \t^* 
\in
 \arg\max_{\t\in\cT} \ol\cE'_L[f],\q\mbox{for}~f\colon\Th\rightarrow\dbR
\end{equation}
In particular, the optimal probability $\dbP^*$ for the first optimization exists if $f$ is $d_\infty$-u.s.c.\ and bounded from above (see Section \ref{sec:preliminary}), while in \cite{ETZ-os} the authors proved for the second optimization the optimal stopping time $\t^*$ exists if $f$ is bounded $d_\infty$-uniformly continuous. This change allows the viscosity solution under the new definition to have better properties. For example, it allows us to prove a stronger stability of solutions (see Section \ref{sec:stability}) and a comparison result for semicontinuous solutions (see Section \ref{sec:comparison}). 

%

Finally, the change of definition does not threat the already proven results in the path-dependant PDE literature, namely, comparison \cite{ETZ2, Ren, Ren2015}, convergence of numerical schemes \cite{RT}, etc. In fact, the arguments for these results will stay in the same lines, 
while necessary modifications need to be made concerning the optimization problems  in \eqref{differentoptimization}.
\end{remark}

\section{Stability}\label{sec:stability}

For all $\dbG$-stopping time $\t\le T$, define the family:
\begin{equation*}
\cP_L(\t,\vartheta) \coloneqq  \Big\{\dbP\in \cP_L\colon \T \le T-\t(\vartheta), \ \mathbb{P}\mbox{-a.s.}\Big\}.
\end{equation*}
We list the following important properties concerning the families of probabilities 
$\mathcal{P}_L$ and $\mathcal{P}_L(\tau,\vartheta)$.
 The proofs are postponed to Appendix.

\begin{proposition}\label{2017-07-18:26}
  The set $\mathcal{P}_L$ is compact.
\end{proposition}

\begin{lem}\label{lem:graphclose}
Let $\t\colon\tilde\Th\rightarrow [0, T]$ be a continuous $\dbG$-stopping time. Then the graph of the function 
\begin{equation*}
\tilde{\Theta}\rightarrow\mathcal{P},\  \vartheta\mapsto \cP_L(\t,\vartheta) 
\end{equation*}
is closed, i.e., given $\{(\vartheta^n,\dbP^n)\}_{n\in\dbN}, (\vartheta^*,\dbP^*)$ such that  $\dbP^n\in \cP_L(\t,\vartheta^n)$ for each $n$, $|\vartheta^n-\vartheta^*|_\infty\rightarrow 0$ and $\dbP^n\rightarrow\dbP^*$, we have $\dbP^*\in \cP_L(\t,\vartheta^*)$.
\end{lem}

Define the nonlinear (conditional) expectations:
\begin{equation*}
  \overline{\mathcal{E}}^\t_L [\cd] (\vartheta)\coloneqq \sup_{\mathbb{P}\in \mathcal{P}_L(\t,\vartheta)}\dbE^\dbP[\cd]
\qquad\qquad
  \underline{\mathcal{E}}^\t_L [\cd](\vartheta) \coloneqq \inf_{\mathbb{P}\in \mathcal{P}_L(\t, \vartheta)}\dbE^\dbP[\cd].
\end{equation*}

\begin{prop}\label{prop:dpp}
Let $f\colon\Th \rightarrow \dbR$ be $d_\infty$-u.s.c.\  bounded from above, and 
$\t$ be a continuous $\dbG$-stopping time. We have
\begin{equation}\label{2018-05-30:12}
\ol\cE_L\Big[ f 1_{\{\T<\t\}}+\ol\cE^\t_L \big[f^{\theta}\big]
_{|
\theta=(\t,\B)}1_{\{\t\le\T\}
}
 \Big] 
=
 \ol\cE_L\big[f\big].
\end{equation}
In particular we have
\begin{equation}\label{2018-05-30:13}
\ol\cE_L\Big[ \ol\cE^\T_L\big[ f^{\theta} \big]_{|\theta=(\T,\B)} \Big] 
=  \ol\cE_L \big[f \big].
\end{equation}
Moreover, let $\dbP^*\in \cP_L$  
be such that $\ol\cE_L[f] = \dbE^{\dbP^*}[f]$. Then we have
\begin{equation*}
  f  =  \ol\cE^\T_L\big[ f^{\T,\B} \big]\q \mbox{$\dbP^*$-a.s.}
\end{equation*}
\end{prop}

As explained in Remark \ref{rem:comparedef}, in this section we will exploit the advantage of the new definition of the jets in order to prove a better stability result. 
As we will see in the rest of the paper, one may apply many pseudo-metrics on $\Th$ other than $d_\infty$. 


In what follows, we denote by $d$ a pseudo-metric on 
$\Theta $ 
which is $d_\infty$-continuous
and such that
 $d(\theta,\theta')=0$
implies
 $t=t'$
and
 $\omega_{t\wedge \cdot}=\omega'_{t\wedge \cdot}$.
In order to recall 
that $d$ has these properties, we will often write $d\ll d_\infty$.





\begin{prop}\label{prop:stable}
Let $d$ be a pseudo-metric defined on $\Th$ such that $ d \ll d_\infty$.
 Let $u_n$ be a sequence of $d_\infty$-u.s.c.\ functions uniformly bounded from above. 
Define the function $u\colon\Th\rightarrow \dbR$ by
\begin{equation*}
  u(\th) \coloneqq 
\limsup_{
\substack{
d(\hat{\theta},\th)\rightarrow 0\\
n\rightarrow \infty}} u_n(\hat{\th}).
\end{equation*}
Then, for any $(\a,\b,\g) \in \ul\cJ_L u(\th)$, there exist $\hat\th_n \in \Th$
and $(\a_n,\b_n, \g)\in \ul\cJ_L u_n(\hat\th_n)$ such that
\begin{equation*}
d(\hat\th_n ,\th )\rightarrow 0,\q \big( u_n(\hat\th_n),\a_n,\b_n\big) 
\rightarrow \big(u(\th), \a,\b\big).  
\end{equation*}
\end{prop}

\begin{proof}
\no\underline{\emph{Step 1.}}
 Since
 $(\alpha,\beta,\gamma)\in \underline{\mathcal{J}}_L u(\theta)$, we also have $(\alpha+\epsilon,\beta,\gamma)\in \underline{\mathcal{J}}_L u(\theta)$ for any $\epsilon>0$. Let $\th_n$ be a sequence such that
 \begin{equation*}
 d(\th_n,\th)\rightarrow 0\q\mbox{and}\q u(\th) = \lim_{n\rightarrow \infty} u_n(\th_n).
 \end{equation*}
  For the simplicity of notation, we denote
\begin{equation*}
U (\th')\coloneqq  \big(u^{\theta}-\varphi^{\alpha+\epsilon,\beta,\gamma}\big) \big(t'\we \ch_\d(\th'), \o' \big)
,\q U_n (\th')\coloneqq  \big(u_n^{\theta_n}-\varphi^{\alpha+\epsilon,\beta,\gamma}\big) \big(t'\we \ch_\d(\th'), \o' \big),\q\mbox{for all}~\th'\in \Th.
\end{equation*}
Since  $\mathcal{P}_L$ is compact (Proposition~\ref{2017-07-18:26})
 and the maps 
$\mathcal{P}\rightarrow \mathbb{\overline R}, \mathbb{P} \mapsto 
  \mathbb{E}^{\mathbb{P}}
   \left[ U_n \right] $
are upper semicontinuous, there exists $\mathbb{P}_n\in \mathcal{P}_L$ for each $n\in \dbN$ such that 
\begin{equation}
  \label{eq:2017-07-18:25}
  \mathbb{E}^{\mathbb{P}_n}
   \big[ U_n \big] 
    = 
\overline{\mathcal{E}}_L
   \big[ U_n \big].
\end{equation}
By considering a subsequence if necessary, again denoted by $\mathbb{P}_n$, we can find
 $\mathbb{P}^*\in \mathcal{P}_L$ such that $\mathbb{P}_n\rightarrow \mathbb{P}^*$.
 By Skorohod's representation there exists a probability space $(\widehat \O, \widehat \cG, \widehat \dbP)$ on which
\begin{equation}\label{skorohodrep}
\exists\q \mbox{r.v.'s}\q \X^n \Big|_{\widehat \dbP} \stackrel{d}{=} \X \Big|_{\dbP_n},\q
\X^* \Big|_{\widehat \dbP} \stackrel{d}{=}\X \Big|_{\dbP^*}\q \mbox{such that}\q \big|\X^n- \X^*\big|_\infty\rightarrow 0\q \widehat \dbP\mbox{-a.s.,}
\end{equation}
where $ \stackrel{d}{=} $ denotes equality in distribution.
Together with \eqref{eq:2017-07-18:25} we obtain
\begin{equation}\label{mainestimate-stable}
\begin{split}
      u(\theta)&= \lim_{n\rightarrow \infty}u_n(\theta_n)
   \leq 
  \limsup_{n\rightarrow \infty} \overline{\mathcal{E}}_L  \big[ U_n  \big]
 = 
\limsup_{n\rightarrow \infty}      
\mathbb{E}^{\mathbb{P}_n}
   \big[ U_n \big] 
 = 
\limsup_{n\rightarrow \infty}      
\mathbb{E}^{\widehat \dbP}
   \left[ U_n(\X^n) \right]    \\
&\le      
\mathbb{E}^{\widehat \dbP}
   \left[\limsup_{n\rightarrow \infty}  U_n(\X^n) \right]
   \le  
\mathbb{E}^{\widehat\dbP}
   \left[ U(\X^*) \right] 
    = 
\mathbb{E}^{\mathbb{P}^*}
   \big[ U \big] 
\leq
 \overline{\mathcal{E}}_L
   \big[ U \big]
    =  u(\theta),
  \end{split}
\end{equation}
The first inequality in the last  line is due to Fatou's lemma,
 the second one is due to the definition of $u$,
whereas the last 
equality  is due to the assumption $(\alpha,\beta,\gamma)\in \mathcal{\underline J}_Lu(\theta)$  and by choosing any $\delta$ sufficiently small.
Therefore
\begin{equation*}
\mathbb{E}^{\mathbb{P}^*}
   \big[ U \big] 
 = 
\overline{\mathcal{E}}_L
   \big[ U \big],
\end{equation*}
which, together with $(\alpha,\beta,\gamma)\in \overline{\mathcal{J}}_L u(\theta)$, provides
\begin{equation}\label{2017-07-18:27}
  \mathbb{P}^*[\T=0]=1.
\end{equation}
Since $\ch_\d$ is continuous,
from
\eqref{2017-07-18:27}
and by denoting $\T^*=\T(\X^*),\ \T^n=\T(\X^n)$,
 we have
\begin{equation}\label{alwasyinball}
   1 =
\mathbb{\widehat
P}\left[
\T^*=0\right]=
 \widehat\dbP \big[\T^* <\ch_\d(\X^*) \big] = \widehat\dbP \big[ \lim_{n\rightarrow\infty} (\T^n -\ch_\d(\X^n))<0 \big]
 \le  \widehat\dbP \big[ \liminf_{n\rightarrow\infty} \big\{\T^n <\ch_\d(\X^n)\big\}\big].
\end{equation}
Further, 
denoting $\B^*=\B(\X^*)$
and $\B^n=\B(\X^n)$,
by \eqref{skorohodrep},\eqref{2017-07-18:27}
and by 
$d_\infty$-continuity of $d$,
 we have
\begin{equation}
  \label{alwasyinball2}
  \begin{split}
    1 &= \widehat\dbP\Big[ d\big(\th, (t+\T^*, \o\otimes_t \B^*)\big) <\d \Big]  =  \widehat\dbP\Big[ \lim_{n\rightarrow\infty} d\big(\th_n, (t_n+\T^n, \o_n \otimes_{t_n} \B^n)\big) <\d \Big] \\
&\le  \widehat\dbP\Big[ \liminf_{n\rightarrow\infty} \Big\{d\big(\th_n, (t_n+\T^n, \o_n \otimes_{t_n} \B^n)\big) <\d \Big\} \Big].
\end{split}
\end{equation}
\no\underline{\emph{Step 2.}}
It follows from \eqref{eq:2017-07-18:25} and Proposition \ref{prop:dpp} that
\begin{equation*}
  U_n  =  \ol\cE_L^{\T}\big[ U_n^{\T,\B} \big],\q\mbox{$\dbP_n$-a.s.}
\end{equation*}
Therefore,  for all $n\in \dbN$,
\begin{equation}
  \label{alwaysoptimal}
\dbP_n\big[
 \Xi^n\big] =1, \q\mbox{where}\q
\Xi^n \coloneqq  \left\{ \vartheta' \in \widetilde\Theta\colon  U_n(\th')  = \ol\cE_L^{t'}
   \big[ U_n^{\th'}\big] (\vartheta') \right\}.
\end{equation}
Note that
\begin{equation*}
\Xi^n\cap \{\T<\ch_\d\} 
\subset
\Big\{
\vartheta'\in \widetilde\Theta
\colon
 \big(\a+\e, \b + 2\g \o'_{t'} ,\g \big)\in \ul\cJ_L u_n^{\th_n} (\th') \Big\},
\end{equation*}
and by \eqref{alwasyinball}, \eqref{alwasyinball2} and \eqref{alwaysoptimal} we have

\be\label{jetapproximate}
\widehat\dbP\Big[ \liminf_{n\rightarrow \infty}  \Big\{ \big(\a+\e, \b + 2\g \B^n_{\T^n} ,\g \big)\in \ul\cJ_L u_n^{\th_n}  (\T^n, \B^n),~ d\big(\th_n, (t_n+\T^n, \o_n \otimes_{t_n} \B^n)\big) <\d \Big\}  \Big]  =  1 .
\ee

\ms

\no\underline{\emph{Step 3.}}
It follows from \eqref{2017-07-18:27} that
\begin{equation*}
\widehat \dbP \big[ \lim_{n\rightarrow\infty }\T^n =0\big] =1.
\end{equation*}
Together with \eqref{mainestimate-stable} we obtain
\begin{equation*}
u(\th)
=
\dbE^{\widehat \dbP} \Big[ \limsup_{n\rightarrow\infty} U_n(\X^n) \Big] 
=
\dbE^{\widehat \dbP} \Big[ \limsup_{n\rightarrow\infty} u_n^{\th_n}(\T^n, \B^n) \Big].
\end{equation*}
By the definition of $u$ we know that $\limsup_{n\rightarrow\infty} u_n^{\th_n}(\T^n, \B^n) \le u(\th)$ $\widehat\dbP$-a.s. Therefore,
\begin{equation*}
\widehat\dbP\Big[ \limsup_{n\rightarrow\infty} u_n^{\th_n}(\T^n, \B^n) =u(\th)\Big] =1.
\end{equation*}
Together with \eqref{jetapproximate}, we obtain
\begin{equation}
\begin{multlined}[c][0.92\displaywidth] \label{setcontainall}
\widehat\dbP\Big[ \Big\{\limsup_{n\rightarrow\infty} u_n^{\th_n}(\T^n, \B^n) =u(\th)\Big\}  \\
\cap \liminf_{n\rightarrow\infty}  \Big\{ \big(\a+\e, \b + 2\g \B^n_{\T^n} ,\g \big)\in \ul\cJ_L u_n^{\th_n}  (\T^n, \B^n),~ d\big(\th_n, (t_n+\T^n, \o_n \otimes_{t_n} \B^n)\big) <\d \Big\}  \Big] =1.
\end{multlined}
\end{equation}
Denote $\hat\th_n\coloneqq  \big(t_n+\T^n(\vartheta'), \o_n\otimes_{t_n}\B^n(\vartheta') \big)$ and $\b_n\coloneqq \beta + 2\g \B^n_{\T^n}(\vartheta')$ for all $\vartheta'\in \tilde\Th$, $n\in \mathbb{N}$. It follows from \eqref{setcontainall} that
there exist 
 $\vartheta'$
and
$n$  such that
\begin{equation*}
d(\hat\th_{n} ,\th )< d(\th_n,\th) + \d,
\q \big| u_{n}(\hat\th_{n})-u(\th)\big|<
\delta,
\q |\b_{n}-\b|< 2\g \d, \q\mbox{and}\q
(\a+\e, \b_{n},\g)\in \ul\cJ_L u_{n}(\hat\th_{n}).  
\end{equation*}
Finally, since
the constants $\delta,\epsilon$ can be chosen arbitrarily small,
 the desired result follows.
\end{proof}

As direct consequence
of Proposition~\ref{prop:stable} 
we have the following

\begin{theorem}
\label{2018-05-17:00}
Let $d,u_n,u$ be as in
Proposition~\ref{prop:stable}.
Let $G_n\colon \Theta\times \mathbb{R}\times \mathbb{R}^m\times \mathbb{S}^m\rightarrow  \mathbb{\overline R}$ be a sequence of functions. 
Define $G\colon\Theta\times \mathbb{R}\times \mathbb{R}^m\times \mathbb{S}^m\rightarrow  \mathbb{\overline R}$ by
\begin{equation*}
  G(\theta,r,\beta,\gamma)\coloneqq \limsup_{\substack{
n\rightarrow \infty,\ d(\theta_n,\theta)\rightarrow 0\\
r_n\rightarrow r,\
\beta_n\rightarrow \beta
    }}
G_n(\theta_n,r_n,\beta_n,\gamma).
\end{equation*}
If  $u_n$ is a 
$\mathcal{P}_L$-viscosity subsolution 
of
\begin{equation*}
   -\partial_t u_n-
   G_n(\theta,u_n, \partial _\omega u_n, \partial ^2_{\omega\omega}u_n)=0,
\end{equation*}
then 
$u$ is a
$\mathcal{P}_L$-viscosity subsolution 
of
\begin{equation*}
   -\partial_t u-
   G(\theta,u, \partial _\omega u, \partial ^2_{\omega\omega}u)=0.
\end{equation*}
\end{theorem}

\section{Existence of $\cP_L$-viscosity solution:
Perron's method}\label{sec:perron}

As we see in the classical literature on the viscosity solutions, one can apply the stability and the comparison results for  semi-continuous solutions to prove the existence of viscosity solution via the so-called Perron's method. In Section \ref{sec:stability}, we have established a quite general stability result
(Theorem~\ref{2018-05-17:00}), and we will leave the discussion on the comparison result to the next sections. In this section, we adapt the Perron's method to the context of $\cP_L$-viscosity solution, assuming some comparison result holds true.

\begin{assumptions}\label{2017-02-21:05}{$\ $}
  Let $d$ be a pseudo-metric on $\Th$ such that $d\ll d_\infty$. 
  \begin{enumerate}[(i)]
  \item \label{2018-03-16:01}
$G\colon 
\Theta\times \mathbb{R}\times \mathbb{R}^m\times \mathbb{S}^m\rightarrow \mathbb{R}$
is 
$d+|\cdot|+|\cdot|_\infty+|\cdot|_\infty$-uniformly continuous
on
$d+|\cdot|+|\cdot|_\infty+|\cdot|_\infty$-bounded sets.
  \item\label{2018-05-26:02} For every $(\theta,\beta,\gamma)\in
\Theta\times \mathbb{R}^m\times \mathbb{S}^m$, the function $r \mapsto  G(\theta,r,\beta,\gamma)$ is non-decreasing.
  \item\label{2018-03-16:02}
 For every $(\theta,r,\beta)\in
\Theta\times \mathbb{R}\times \mathbb{R}^m$,
    \begin{equation*}
      G(\theta,r,\beta,\gamma)\leq    
   G(\theta,r,\beta,\gamma')\qquad \forall \gamma,\gamma'\in \mathbb{S}^m,\ \gamma\leq \gamma'.
    \end{equation*}
  \item\label{2018-03-20:05}
 For all $(\theta,r)\in \Theta\times \mathbb{R}$, $\beta,\beta'\in \mathbb{R}^m$, $\gamma,\gamma'\in \mathbb{S}^m$,
    \begin{equation*}      
 |G(\th,r, \beta+\beta', \gamma+\gamma')-G(\th,r,\beta,\gamma)| \le L \big(|\beta'| + |\gamma'|\big).
\end{equation*}
  \end{enumerate}
\end{assumptions}

\begin{assum}\label{2018-03-16:00}
Let $d$ be a pseudo-metric on $\Th$ such that $d\ll d_\infty$. Assume that there is a bounded $d$-u.s.c.\  viscosity subsolution $\ul u$ and a bounded $d$-l.s.c.\  viscosity  supersolution $\ol v$ of PPDE~\eqref{eq:PPDEgeneral} such that
\begin{enumerate}[i)]
\item  $\ul u \le \ol v$ on $\Th$; 
if $\ul u_*$
denote
 the $d$-l.s.c.\
 envelop of $\ul u$,
 then   $\ul u_{*}(T,\cdot) = \ol v(T,\cdot) \eqqcolon\xi$.
\item For any $d$-u.s.c.\  viscosity subsolution $u$ and $d$-l.s.c.\  viscosity  supersolution $v$ of PPDE~\eqref{eq:PPDEgeneral} taking values between $\ul u_*$ and $\ol v$, we have $u \le v$ on $\Th$.
\end{enumerate}.
\end{assum}

\begin{thm}\label{thm: Perron}
Let Assumptions~\ref{2017-02-21:05} and \ref{2018-03-16:00} hold true. Denote
\begin{equation*}
\cD \coloneqq  
\big\{\phi\colon\Th\rightarrow \dbR \colon \phi~\mbox{is a bounded $d$-u.s.c.\ $\cP_L$-viscosity subsolution of  \eqref{eq:PPDEgeneral} and}~\ul u\le \phi\le \ol v \big\}.
\end{equation*}
Then $u(\th)\coloneqq \sup\big\{\phi(\th)\colon\phi\in\cD\big\}$ 
is a $d$-continuous $\cP_L$-viscosity solution of  \eqref{eq:PPDEgeneral}, and satisfies the boundary condition $u(T,\cdot)=\xi$.
\end{thm}

Before proving Theorem \ref{thm: Perron}, we show some useful lemmas. The following proposition is a direct corollary of the stability result in Proposition \ref{prop:stable}.

\begin{prop}\label{prop:stableBnorm}
Let $\ol u$ be $d_\infty$-u.s.c., $\th\in \Th$ and $(\a,\b,\g) \in \ul\cJ_L \ol u(\th)$. Suppose also that $u_n$ is a sequence of $d_\infty$-u.s.c.\ functions uniformly bounded from above such that
\begin{equation*}
\begin{cases}
(1) & \mbox{there exist $\th_n\in\Th$ such that}~ d(\th_n,\th) \rightarrow 0, ~u_n(\th_n)\rightarrow \ol u(\th),\\
(2) & \mbox{if $\th'_n\in \Th$ and $d(\th'_n , \th') \rightarrow 0$, then}~ \limsup_{n\rightarrow\infty} u_n(\th'_n) \le \ol u(\th').
\end{cases}
\end{equation*}
Then there exist $\hat\th_n \in \Th$, $(\a_n,\b_n, \g)\in \ul\cJ_L u_n(\hat\th_n)$ such that
\begin{equation*}
d(\hat\th_n ,\th )\rightarrow 0,\q \big( u_n(\hat\th_n),\a_n,\b_n\big) \rightarrow \big(\ol u(\th), \a,\b\big).  
\end{equation*}
\end{prop}

\begin{proof}
Recall the definition of $u$ in Proposition \ref{prop:stable}. We clearly have
\begin{equation*}
  \ol u(\th)  =  u(\th)\q\mbox{and}\q
\ol u(\th')  \ge  u(\th') \q\mbox{for all}\q \th'\in \Th.
\end{equation*}
It the follows that $(\a,\b,\g)\in \ul \cJ_L u (\th)$. Then the desired result follows from Proposition \ref{prop:stable}.
\end{proof}

For $\gamma\in \mathbb{S}^m$, we define
\begin{equation}
  \label{2018-05-26:00}
  (\gamma)^+\coloneqq \sup_{\substack{x\in \mathbb{R}^m,|x|\leq 1}} \langle \gamma x,x\rangle.
\end{equation}

\begin{lemma}\label{2018-06-18:00}
  Let $G$ as in Assumption~\ref{2017-02-21:05}.
Then,
for all 
$ \theta\in\Theta, r\in \mathbb{R}, \beta\in \mathbb{R}^m, \gamma,\gamma'\in \mathbb{S}^m$,
\begin{equation*}
  G(\theta,r,\beta,\gamma)
  -
  G(\theta,r,\beta,\gamma')\leq 
  (\gamma-\gamma')^+.
\end{equation*}
\end{lemma}
\begin{proof}
  By definition of $(\cdot)^+$ in
  \eqref{2018-05-26:00}, we have
    $ \gamma-\gamma'\leq L(\gamma-\gamma')^+$.
By
the ellipticity condition
Assumption~\ref{2017-02-21:05}\eqref{2018-03-16:02}
and the Lipschitz
continuity condition in
Assumption~\ref{2017-02-21:05}\eqref{2018-03-20:05}
 we then have
\begin{equation*}
  G(\theta,r,\beta,\gamma)
=
  G(\theta,r,\beta,\gamma-\gamma'+\gamma')
\leq
  G(\theta,r,\beta,(\gamma-\gamma')^++\gamma')
\leq
L(\gamma-\gamma')+
  G(\theta,r,\beta,\gamma'),
\end{equation*}
which concludes the proof.
\end{proof}

\begin{lem}\label{lem:puccioperator}
Recall the definition of the nonlinear expectation $\ul\cE'_L$ in Remark \ref{rem:comparedef}. Let $\xi\colon\O\rightarrow\dbR$ be bounded, $|\cd|_\infty$-u.s.c. Define 
\begin{equation*}
\phi(\th)\coloneqq  \ul\cE'_L\left[\xi^\th\right]
 + \f^{\a,\b,\g}(\th)
. 
\end{equation*}
Then $\phi$ is a  locally bounded  $d_\infty$-u.s.c.\ function and is a $\cP_L$-viscosity subsolution to 
\begin{equation}
  \label{puccieq}
-\pa_t \phi +\a +L |\b+\g\o_t-\pa_\o \phi | +L (\g - \pa^2_{\o\o} \phi)^+  =  0.
\end{equation}
\end{lem}
\begin{proof}
It is clear that $\phi$ is a  locally bounded  $d_\infty$-u.s.c.\ function.
Define 
\begin{equation*}
  \psi(\th) \coloneqq  \phi(\th)-\f^{\a,\b,\g}(\th)
=
\ul\cE'_L \left[ \xi^\theta \right] 
.
\end{equation*}
By the dynamic programming result (see e.g.\ Theorem 2.3 in \cite{NvH}), we have
\begin{equation*}
\psi(\th) =  \ul\cE'_L\Big[\psi^\th(\eta,\B) \Big]\q\mbox{for all $\dbF$-stopping time $\eta\ge 0$}. 
\end{equation*}
Then it is easy to verify $\psi$ is a $\cP_L$-viscosity subsolution to
\begin{equation}\label{2018-05-26:01}
-\pa_t \psi +L |\pa_\o \psi|
+L (-\pa^2_{\o\o} \psi)^+  =  0.
\end{equation}
and, using
\eqref{2018-05-26:01},
it is not 
difficult too see that
 $\phi$ is a $\cP_L$-viscosity subsolution to \eqref{puccieq}.
\end{proof}

\begin{lem}\label{lem:selectionlsc}
 Let $\xi\colon\O\rightarrow\dbR$ be bounded and $d_\infty$-l.s.c. Define 
\begin{equation*}
\phi(\th)\coloneqq  \ul\cE'_L\Big[\xi^\th \Big]. 
\end{equation*}
Then $\phi$ is a  bounded  $d_\infty$-l.s.c.\ function.
\end{lem}
\begin{proof}
Define the function
\begin{equation*}
f (\th, \dbP) \coloneqq  \dbE^\dbP[\xi^\th],\q\mbox{for all $\th\in \Th$ and $\dbP\in \cP'_L$}.
\end{equation*}
Let $\th_n\in \Th$ and $\dbP_n\in \cP'_L$ such that $d_\infty(\th_n,\th)\rightarrow 0$ and $\dbP_n\rightarrow\dbP$. By Skorohod's representation there exists a probability space $(\widehat \O, \widehat \cG, \widehat \dbP)$ on which
\begin{equation*}
\exists\q \mbox{r.v.'s}\q \X^n \Big|_{\widehat \dbP} \stackrel{d}{=} \X \Big|_{\dbP_n},\q
\X^* \Big|_{\widehat \dbP} \stackrel{d}{=}\X \Big|_{\dbP}\q \mbox{such that}\q \big|\X^n- \X^*\big|_\infty\rightarrow 0, \q \widehat \dbP\mbox{-a.s.}
\end{equation*}
 So we have
\begin{equation*}
\liminf_{n\rightarrow\infty} f(\th_n,\dbP_n)
 =  \liminf_{n\rightarrow\infty} \dbE^{\widehat\dbP} \Big[ \xi\big(\o_n\otimes_{t_n} \B^n\big)\Big]
 \ge  \dbE^{\widehat\dbP} \Big[ \xi\big(\o\otimes_{t} \B^*\big)\Big]  =  f(\th,\dbP),
\end{equation*}
where the inequality is due to Fatou's lemma and the $d_\infty$-l.s.c.\ of $\xi$. Therefore, $f$ is l.s.c.\ on $\Th\times \cP'_L$. Further, note that
\begin{equation*}
\phi(\th) = \sup_{\dbP\in \cP'_L} f(\th,\dbP)
\end{equation*}
and $\cP'_L$ is compact
 (see \cite{Meyer1984}). By Proposition 7.32 of \cite{BS}, we obtain that $\phi$ is $d_\infty$-l.s.c.
\end{proof}

\begin{lem}\label{lem:inftysolenvelop}
Let Assumption \ref{2017-02-21:05} hold true. Let $u$ be a
bounded
 $\cP_L$-viscosity subsolution (resp.\ supersolution) and $u^*$ (resp.\ $u_*$) be its $d$-u.s.c.\ (resp.\ l.s.c.) envelop. Then $u^*$ (resp.\ $u_*$) is also a $\cP_L$-viscosity subsolution (resp.\ supersolution).
\end{lem}
\begin{proof}
We only show the result for subsolution. The one for the supersolution follows from the same argument.
First, notice that $d\ll d_\infty$ implies that $u^*$ is $d_\infty$-u.s.c.
Now let
 $\th\in \Th$ and $(\a,\b,\g) \in \ul\cJ   u^*(\th)$.
By the very definition of $u^*$, it is clear that the assumptions of
Proposition~\ref{prop:stableBnorm}
are fulfilled (with the functions $u_n,\overline u$
appearing in the statement replaced
 by $u,u^*$, respectively).
Hence there exist $\hat\th_n \in \Th$, $(\a_n,\b_n, \g)\in \ul\cJ   u(\hat\th_n)$ such that
\begin{equation*}
d(\hat\th_n , \th )\rightarrow 0,\q \big(   u(\hat\th_n),\a_n,\b_n\big) \rightarrow \big(  u^*(\th), \a,\b\big).  
\end{equation*}
Since $  u$ is a $\cP_L$-viscosity subsolution to \eqref{eq:PPDEgeneral}, we have
\begin{equation*}
-\a_n -G(\hat\th_n,   u(\hat\th_n), \b_n, \g) \le 0. 
\end{equation*}
By letting $n\rightarrow\infty$, we have $-\a -G(\th,   u^*(\th), \b, \g) \le 0$. 
\end{proof}

\begin{lem}\label{lem:shiftedsetD}
For $\theta\in \Theta$, with $t<T$,  define 
\begin{equation*}
  \begin{split}
    \cD(\th) \coloneqq 
 \big\{\phi\colon[0,T-t]\times
 \O\rightarrow \dbR 
\colon
&\ul u^\th \le \phi\le \ol v^\th
~\mbox{and}
\\
&\phi~\mbox{is a
bounded
 $d$-u.s.c.\ viscosity subsolution of  \eqref{eq:PPDEgeneral} on $[0,T-t]$}
 \big\},
\end{split}
\end{equation*}
and $\tilde u (\th') \coloneqq  \sup\{\phi(\th')\colon\phi\in\cD(\th)\}$. Then we have $u(\th) = \tilde u(0)$. 
\end{lem}
\begin{proof}
First it is obvious that, for any $\phi \in \cD$, we have $\phi^\th\in \cD(\th)$. So $u(\th)\le \tilde u(0)$.

On the other hand, suppose $u(\th)<\tilde u(0)$. Then there is $\phi\in 
\cD(\th)$ such that $\phi(0)> u(\th)$. Now take any $\tilde\phi\in \cD$, and define, for $\theta'\in \Theta$,
\begin{equation*}
\Phi (\th')\coloneqq 
\begin{dcases}
\tilde\phi(\th')\vee \phi(\hat\th)&\mbox{if $\hat t+t=t'\ge t$ and $\o' = \o \otimes_t \hat\o$},\\
\tilde\phi(\th')& \mbox{elsewhere.}
\end{dcases}
\end{equation*}
It is easy to verify that $\Phi$ is 
$d_\infty$-u.s.c.\ 
and that,
 by the viscosity subsolution property of $\tilde\phi$ and $\phi$,
$\Phi$ is a 
 $\cP_L$-viscosity subsolution to \eqref{eq:PPDEgeneral}.
 Then, it follows from Lemma \ref{lem:inftysolenvelop} that the $d$-u.s.c.\ envelop $\Phi^*$ is also a $\cP_L$-viscosity subsolution to \eqref{eq:PPDEgeneral}.
 Finally note that
\begin{equation*}
\Phi^*(\th)  \ge  \Phi(\th)  \ge  \phi(0)>u(\th),
\end{equation*}
which is a contradiction to the definition of $u$. Therefore, we have $u(\th)= \tilde u(0)$.
\end{proof}

\begin{proof}[\textbf{Proof of Theorem \ref{thm: Perron}}]
 Thanks to Lemma~\ref{lem:shiftedsetD}, we only need to
check the property of viscosity solution at the point $\th = (0,0)$.
\ms
 
\no\underline{\emph{Step 1.}} We first prove that $u$ is 
 $d$-u.s.c.\ and that it is a 
 $\cP_L$-viscosity subsolution. Let $(\a,\b,\g)\in \mathcal{\underline J}u^*(0,0)$, 
where $u^*$ is the $d$-u.s.c.\ envelop of $u$.
By the definition of $u$ and $u^*$, there exists a sequence of $u_n \in \cD$ and $\th_n\in \Th$ such that
\begin{equation}\label{2018-05-30:00}
   d(\th_n,(0,0)) \rightarrow 0, ~u_n(\th_n)\rightarrow u^*(0,0).
\end{equation} 
Moreover,
\begin{equation}
  \label{eq:2018-05-30:01}
  \mbox{if $\th'_n\in \Th$ and $d(\th'_n , \th') \rightarrow 0$, then}~ \limsup_{n\rightarrow\infty} u_n(\th'_n) \le u^*(\th').
\end{equation}
It follows from Proposition \ref{prop:stableBnorm} that there exist $\hat\th_n \in \Th$ and $(\a_n,\b_n, \g)\in \ul\cJ u_n(\hat\th_n)$ such that
\begin{equation*}
d(\hat\th_n , (0,0) )\rightarrow 0,\q \big( u_n(\hat\th_n),\a_n,\b_n\big) \rightarrow \big(u^*(0,0), \a,\b\big).  
\end{equation*}
Since $u_n$ are all $\cP_L$-viscosity subsolution to \eqref{eq:PPDEgeneral}, we have
\begin{equation*}
-\a_n -G(\hat\th_n, u_n(\hat\th_n), \b_n, \g) \le 0. 
\end{equation*}
If we let $n\rightarrow\infty$, we
obtain  $-\a -G\big((0,0), u^*(0,0), \b, \g \big) \le 0$. Therefore,  $u^*$ is a $\cP_L$-viscosity subsolution. By definition of $u$, we have $u^*\le u$. On the other hand, by definition of $u^*$, we also have $u^*\ge u$. So finally we conclude that
\begin{equation*}
  u=u^*
\end{equation*}
is $d$-u.s.c.\ and  a $\cP_L$-viscosity subsolution. 

\ms

\no\underline{\emph{Step 2.}} Let $\tilde u_*$ be the $d_\infty$-l.s.c.\ envelop of the function $u$. We now prove $\tilde u_*$ is a $\cP_L$-viscosity supersolution at $\theta=0$. Suppose it is not the case and there is $(\a,\b,\g)\in \ol\cJ \tilde u_*(0,0)$ such that
\begin{equation*}
-\a -G\big((0,0), \tilde u_*(0,0), \b,\g\big) = -3\e <0.
\end{equation*}
Therefore, for $\d$ small enough we have
\begin{equation}\label{smalldeltaobj}
    \begin{gathered}
\tilde u_*(0,0) = \ul\cE_L \Big[(\tilde u_* - \f^{\a,\b,\g})(\T\we\ch_\d, \B) \Big],\\
\d \le \frac{\e}{L|\g|} \q\mbox{and}\q -\a -G\big(\th, \tilde u_*(\th), \b+\gamma \omega_t,\g\big) <- 2\e
\q\mbox{for all $t \le \ch_\d(\o)$}.
\end{gathered}
\end{equation}
For the simplicity of notation, we denote for all $\th\in \Th$
\begin{equation*}
 \xi(\th)\coloneqq  u\big(\ch_\d(\o),\o\big) -\f^{\a-\e,\b,\g}\big(\ch_\d(\o),\o\big) \q\mbox{and}\q \ul \xi(\th)\coloneqq  \tilde u_*\big(\ch_\d(\o),\o\big)-\f^{\a-\e,\b,\g}\big(\ch_\d(\o),\o\big).
\end{equation*}
 Recall the nonlinear expectation $\ul \cE'_L$ defined in Remark \ref{rem:comparedef}. Define
\begin{equation*}
\phi(\th)\coloneqq  \ul\cE'_L\big[ \xi^{\th} \big] + \f^{\a-\e,\b,\g}(\th),\q
\ul \phi(\th)\coloneqq  \ul\cE'_L\big[ \ul \xi^{\th} \big]+ \f^{\a-\e,\b,\g}(\th).
\end{equation*}
By recalling
 the definitions of
 $\mathcal{\underline E}'_L$ and $\mathcal{\underline E}_L$,  we have
\begin{equation*}
 \ul\phi(0,0) = \ul\cE'_L\big[ \ul\xi\big] \ge
 \ul\cE_L\big[ \ul\xi\big]
=
 \ul\cE_L \big[  (\tilde u_*- \f^{\a-\e,\b,\g})\big(\ch_\d, \B\big) \big] > \tilde u_*(0,0).
\end{equation*}
By Lemma \ref{lem:selectionlsc},
 $\ul\phi$ is $d_\infty$-l.s.c. 
By \emph{Step~1} we know that $u$ is $d$-\ hence $d_\infty$-u.s.c.
Then, 
by Lemma~\ref{lem:puccioperator},
it follows that
$\phi$ is a
locally bounded $d_\infty$-u.s.c.\ function and a $\cP_L$-viscosity subsolution to \begin{equation}\label{2018-06-18:01}
-\pa_t \phi +\a-\e +L |\b +\g\o_t-\pa_\o \phi | +L (\g - \pa^2_{\o\o} \phi)^+  =  0.
\end{equation}
It follows from \eqref{smalldeltaobj} that $\a >2\epsilon - G\big(\th, \tilde u_*(\th), \b+\gamma \omega_t,\g\big) $ on $\big\{\th\colon t< \ch_\d(\o)\big\}$. 
By the Lipschitz
continuity assumptions on $G$
and
Lemma~\ref{2018-06-18:00},
we then have that 
  $\phi$ is a $\cP_L$-viscosity subsolution to
\begin{equation*}
-\pa_t \phi -G(\cd, \tilde u_*, \pa_\o \phi, \pa_{\o\o}^2 \phi)  =  0,\q\mbox{on}\q \big\{\th\colon t< \ch_\d(\o)\big\}.
\end{equation*}
Further, since $\tilde u_*\le u\le u\vee \phi$ and $G$ is 
non-decreasing in $r$, the function $\phi$ is a $\cP_L$-viscosity subsolution to
\begin{equation}
  \label{eq:subsolphi}
-\pa_t \phi -G(\cd, u\vee \phi, \pa_\o \phi, \pa_{\o\o}^2 \phi)  =  0,\q\mbox{on}\q \big\{\th\colon t< \ch_\d(\o)\big\}.
\end{equation}
Now define
\begin{equation*}
U \coloneqq  ( u\vee \phi )1_{\{t<\ch_\d(\o)\}} + u 1_{\{t\ge\ch_\d(\o)\}}.
\end{equation*}
Recall that $u,\phi$ are both $d_\infty$-u.s.c.\ and that $\ch_\d$ is continuous,
and observe that $u=\phi$ on $\{t=\ch_\d(\o)\}$. So $U$ is $d_\infty$-u.s.c. Further, since $u$ is a $\cP_L$-viscosity subsolution to \eqref{eq:PPDEgeneral} and $\phi$ is a $\cP_L$-viscosity subsolution to \eqref{eq:subsolphi} on $\{t<\ch_\d(\o)\}$, it is easy to verify that  $U$ is a $\cP_L$-viscosity subsolution to \eqref{eq:PPDEgeneral}. Then it follows from Lemma \ref{lem:inftysolenvelop} that the $d$-u.s.c.\ envelop of $U$, namely $U^*$, is a  $\cP_L$-viscosity subsolution to \eqref{eq:PPDEgeneral}.

However, by definition of $\tilde u_*$ there is a sequence $\th_n\in \Th$ such that $d_\infty\big(\th_n,(0,0) \big)\rightarrow 0$ and $\tilde u_*(0,0)=\lim_{n\rightarrow\infty} u(\th_n)$, and thus
\begin{equation*}
\liminf_{n\rightarrow\infty} \Big(U^*(\th_n) -u(\th_n)\Big)
\ge \liminf_{n\rightarrow\infty}  \Big( \phi(\th_n) - u(\th_n)\Big)
\ge \liminf_{n\rightarrow\infty} \Big( \ul\phi(\th_n) -u(\th_n)\Big)
\ge \ul\phi(0,0) - \tilde u_*(0,0) >0.
\end{equation*}
The second last inequality is due to the $d_\infty$-l.s.c.\ of the function $\ul\phi$. Therefore, there is $\th_n$ such that $U^*(\th_n) > u(\th_n)$, which is in contradiction with the definition of $u$.

\ms

\no\underline{\emph{Step 3.}} By
\emph{Step 2} and 
 Lemma \ref{lem:inftysolenvelop}, the $d$-l.s.c.\ envelop of $\tilde u_*$, namely $(\tilde u_*)_*$, is a $\cP_L$-viscosity supersolution to \eqref{eq:PPDEgeneral}. Since $\tilde u_* \le u$, we have $(\tilde u_*)_* \le u_*$. On the other hand, since $\tilde u_*\ge u_*$, we have $(\tilde u_*)_* \ge u_*$. Therefore, $u_* = (\tilde u_*)_*$ is a $\cP_L$-viscosity supersolution to \eqref{eq:PPDEgeneral}.

\ms

\no\underline{\emph{Step 4.}}
Note that $ \ul u_*\le u_*\le \ol v$, in particular, $u_*(T,\cdot)=\xi$. By
 \emph{Step~3}
and  Assumption~\ref{2018-03-16:00}, we have $u \le u_*$. Together with the definition of $u_*$ , we have
that $u=u_*$ is 
$d$-l.s.c.\ and a $\mathcal{P}_L$-viscosity supersolution to
\eqref{eq:PPDEgeneral}.
Then, recalling  what proved in 
\emph{Step~1}, we conclude that $u$
is
a $d$-continuous $\cP_L$-viscosity solution to \eqref{eq:PPDEgeneral}.
\end{proof}

\section[Comparison result for semicontinuous solutions]{Comparison result for  ${\overleftarrow d_p}$-semicontinuous solutions}\label{sec:comparison}

We have seen that for  Perron's method the comparison result for semicontinuous solutions is crucial. However, up to now there is no such result for the fully nonlinear path-dependent PDE in the literature. The main difficulty, as explained in Remark \ref{rem:comparedef}, is due to the optimal stopping problem presented in \eqref{differentoptimization}. As an advantage of our modified definition of $\cP_L$-viscosity solutions, we are able to show such a comparison result by combining
the comparison result for uniformly continuous solutions proved in \cite{Ren2015} and the convolution in backward pseudo-metric developed in \cite{Ren-Perron}.

\smallskip
In the present section we will
deal with
the pseudo-metric $\ola d_p$ on $\Theta$ defined, for $1\leq p<\infty$, by
\begin{equation*}
  \ola d_p (\th,\th') \coloneqq  |t-t'| + |\o_t - \o'_{t'}| +  \Big( \int_0^T  |\o_{t-s} -\o'_{t'- s}|^p ds\Big)^{\frac1p}
\qquad \forall \theta,\theta'\in \Theta,
\end{equation*}
where we
extend $\omega,\omega'$ on the negative real line by
 $\o_s =\o'_s =0$ for  $s<0$. 
Notice that $\ola d_p \ll d_\infty$.

\begin{remark}
Let us 
 we recall the pseudo-metric $d_p$ ($1\le p<\infty$) on $\Th$ defined in \cite{Ren2015}:
\begin{equation*}
  d_p(\th,\th') \coloneqq  |t-t'| + \Big( \int_0^{T+1}  |\o_{t\we s} -\o'_{t'\we s}|^p ds\Big)^{\frac1p} 
\qquad \forall \theta,\theta'\in\Theta,
\end{equation*}
where we 
set
 $\omega_s\coloneqq \omega_t, \omega'_s\coloneqq \omega'_{t'}$ for $s\in[T,T+1]$.
It is not difficult to see that
the two pseudo-metrics $\ola d_p,d_p$ induce the same topology on $\Theta$, i.e., the identity map $(\Theta,\ola d_p)\rightarrow (\Theta,d_p)$ is a homeomorphism.
In particular, the classes of upper-/lower-semicontinous functions with respect to $\ola d_p$ and  $d_p$ coincide.
Neverthless,
the pseudo-metric uniformities associated with $\ola d_p$ and with $d_p$ do not coincide, hence 
uniformly continuous functions with respect to $\ola d_p$ may not be uniformly 
continuous with respect to $d_p$, and viceversa.
\end{remark}



The comparison result in this section will hold under the following strong assumption on the nonlinearity $G$.
 
\begin{assum}\label{assum:Gcomp}
The function $G\colon\Th\times \dbR^m\times \dbS^m \rightarrow \dbR$ satisfies the following assumptions.
\begin{enumerate}[(i)]
  \item The function $\th\mapsto G(\th, \beta,\gamma)$ is $\ola d_p$-uniformly continuous, uniformly in $(\beta,\gamma)$.
  
  \item
 For every $(\theta, \beta)\in
\Theta\times \mathbb{R}^m$,
    \begin{equation*}
      G(\theta,\beta,\gamma)\leq    
   G(\theta,\beta,\gamma')\qquad \forall \gamma,\gamma'\in \mathbb{S}^m,\ \gamma\leq \gamma'.
    \end{equation*}
 
  \item
 For all $\theta\in \Theta$, $\beta,\beta'\in \mathbb{R}^m$, $\gamma,\gamma'\in \mathbb{S}^m$,
    \begin{equation*}      
 |G(\th, \beta+\beta', \gamma+\gamma')-G(\th,\beta,\gamma)| \le L \big(|\beta'| + |\gamma'|\big).
\end{equation*}
  \end{enumerate}
\end{assum} 

\no Note that
if Assumption \ref{assum:Gcomp}
holds true then
 Assumption \ref{2017-02-21:05}
holds true with $d=\ola d_p$.

\begin{thm}\label{thm:comparisonuniform}
Let $G$ satisfy Assumption \ref{assum:Gcomp}. Let $u$ (resp.\ $v$) be a bounded $\ola d_p$-uniformly continuous $\cP_L$-viscosity subsolution  (resp.\ supersolution) to PPDE \eqref{eq:PPDEgeneral}. If $u(T,\cdot) \le v(T,\cdot)$,  we have $u\le v$ on $\Th$.
\end{thm}
\begin{proof}
A similar comparison result, in which $u,v,G$ are $d_p$-uniformly continuous, is proved under the old definition of viscosity solution in \cite{Ren2015}. As mentioned in Remark \ref{rem:comparedef}, the change of the definition does not add trouble for proving the existing comparison result. Further, we can indeed apply the  same argument as in \cite{Ren2015} to prove the desired comparison result, where the $d_p$-uniform continuity is replaced by $\ola d_p$-uniform continuity. Since the whole argument is too long, we refer the reader to \cite{Ren2015} for the technical details.
\end{proof}
 
In the rest of this section, we 
 show a comparison result for $\ola d_p$-semicontinuous 
 solutions. 
 The main idea is to approximate the semicontinuous solutions by uniform continuous functions with the following convolution.
For a bounded $\ola d_p$-u.s.c.\ function $u$, we define
\begin{equation}
  \label{eq:supconvolution}
u^n(\th) \coloneqq  \sup_{\th'\in\Th} \Big\{ u(\th') -n \ola d_p(\th,\th') \Big\}.
\end{equation}
Then $u^n$ is a bounded $\ola d_p$-Lipschitz function and $u^n\rightarrow u$ 
pointwise
 as $n\rightarrow \infty$. 

\begin{assum}\label{assum:terminalLip}
For a function $u\colon\Th\rightarrow\dbR$, there is a constant $C_0$ such that
\begin{equation}\label{2018-05-30:02}
  \lim_{\d\rightarrow 0} \sup \left\{\frac{\big| u(\th')-u(T,\o)\big|}{\ola d_p\big(\th', (T,\o)\big)}\colon
\theta'\in \Theta,\ \omega\in \Omega,\ 
\ola d_p\big(\th', (T,\o)\big)\le \d\right\}  \le  C_0,
\end{equation}
where we adopt the convention $\frac00=0$.
\end{assum}

\begin{example}
Let $\ul u, \ol v$ be two $\ola d_p$-Lipschitz functions with Lipschitz constant $C$ and assume that $\ul u(T,\cdot)=\ol v(T,\cdot)$ and $\ul u\le \ol v$ on $\Th$. Then all the functions $u$, such that $\ul u\le u\le \ol v$, satisfy Assumption \ref{assum:terminalLip}.
\end{example}

\begin{lem}\label{lem:unterminalconsistent}
Let $u$ be a bounded $\ola d_p$-u.s.c.\ function satisfing Assumption \ref{assum:terminalLip}. Then we have $u^n(T,\cdot) = u(T,\cdot)$ with $u^n$  defined in \eqref{eq:supconvolution} for $n$ big enough.
\end{lem}
\begin{proof}
Clearly $u^n(T,\cdot)\geq u(T,\cdot)$ for all $n$.
Now let $\delta>0$ be small enough such that the supremum appearing in 
\eqref{2018-05-30:02} is less than $C_0+1$.
Let $n\geq (C_0+1)\vee \frac{2|u|_\infty}{\delta}$.
Then
  \begin{equation*}
    u(T,\omega)=
    u(T,\omega)
    -u(\theta')
    +u(\theta')
    \geq -n\ola d_p(\theta',(T,\omega))+u(\theta')
\qquad 
\forall \theta'\in \Theta, \ \omega\in \Omega,
  \end{equation*}
or, equivalently,
$  u(T,\cdot)\geq u^n(T,\cdot)$.
\end{proof}

The main advantage of the convolution in \eqref{eq:supconvolution} is that $u^n$ inherits the viscosity subsolution property of $u$. We first prove a lemma, which is an adaptation of the important Lemma 3.3 of \cite{Ren2015} to the new definition of $\cP_L$-viscosity solution. Note that in Lemma 3.3 of \cite{Ren2015} the result holds true only for uniformly continuous functions, while the following lemma is proved for u.s.c.\ functions. This improvement is due to our new definition.

\begin{lem}\label{lem:OS}
Let $u$ be a $\ola d_p$-u.s.c.\ function satisfying $u(0,0)>\ol\cE_L \big[(u - \varphi^{\a,\b,\g})(\ch_\d,\B) \big]$, for some $\d>0$ and $(\a,\b,\g)\in\dbR\times\dbR^m\times\dbS^m$. 
Then, there exists $\th^*\in \Theta$ such that
\begin{equation}\label{2018-05-30:05}
   t^*<\ch_\d(\o^*)
 \q\mbox{and}\q
 (\a,\b + \g\o^*_{t^*},\g) \in  \underline{\cJ}_L u(\th^*). 
\end{equation}
\end{lem}

\begin{proof}
Let
 $U \coloneqq  u - \varphi^{\a,\b,\g}$ and
  define the value function $V\colon \Theta\rightarrow \mathbb{R}$ by
  \begin{equation*}
    V(\th) \coloneqq   \ol\cE_L^t \big[ U^\th (\T\we\ch_\d,\B) \big](\vartheta),
  \end{equation*}
where $\vartheta=(t,\omega,a,\mu,q)\in \widetilde\Theta$ is such that $\theta=(t,\omega)$. Notice that the definition of $V(\theta)$ does not depend on the representative $\vartheta$.
By compactness of $\mathcal{P}_L$
and by upper semicontinuity of the map
\begin{equation*}
  \mathcal{P}_L\rightarrow \mathbb{R},\ \mathbb{P} \mapsto
  \mathbb{E}^\mathbb{P}
  \left[ U(\T\wedge \ch_\delta,\B) \right] 
\end{equation*}
we can find
 $\dbP^*\in \cP_L$ such that $V(0,0) = \dbE^{\dbP^*}[U(\T\we\ch_\d,\B)]$. By Proposition \ref{prop:dpp} we have
  \begin{equation}\label{2018-05-30:04}
    U = V \q\dbP^*\mbox{-a.s.}
  \end{equation}
  On the other hand, by 
the assumption of the lemma, we have 
  \begin{equation*}
    \dbE^{\dbP^*}[U(\T\we\ch_\d,\B)] 
 =  V(0,0)  \ge  U(0,0)=u(0,0) 
> \ol\cE_L \big[ U (\ch_\d,\B) \big]
\ge \dbE^{\dbP^*}[U(\ch_\d,\B)] . 
\end{equation*}
Therefore,
\begin{equation}
  \label{2018-05-30:03}
  \dbP^*[\T<\ch_\d]>0.
\end{equation}
By taking into account
\eqref{2018-05-30:04}
and
\eqref{2018-05-30:03},
we conclude that there exists
$\vartheta^*\in \widetilde\Theta$
such that
$t^*  < \ch_\d(\o^*)$ and $U(\th^*)= V(\th^*)$,
which is equivalent
to
\eqref{2018-05-30:05}.
\end{proof}

  If $G$ 
  satisfies Assumption~\ref{assum:Gcomp}, 
  we define the modulus of continuity $\rho_G$ by
  \begin{equation*}
    \rho_G(x)\coloneqq \sup_{\beta\in \mathbb{R}^m, \gamma\in \mathbb{S}^m}
\sup_{\substack{\theta,\theta'\in \Theta\\\ola d_p(\theta,\theta')\leq x}}|G(\theta,\beta,\gamma)-G(\theta',\beta,\gamma)|,\qquad \forall x\geq 0.
  \end{equation*}

\begin{prop}\label{prop:unsubsolution}
Let $u$ be 
a
 $\ola d_p$-u.s.c.\ $\cP_L$-viscosity subsolution to \eqref{eq:PPDEgeneral},
bounded by a constant $C>0$,
and let
$u^n$
be defined as in \eqref{eq:supconvolution}.
Assume that
 $G$ satisfies Assumption~\ref{assum:Gcomp}.
 Then, for $n$ big enough,
$u^n$
is a $\cP_L$-viscosity subsolution to the following equation:
\begin{equation*}
  -\pa_t u^n - G(\th, \pa_\o u^n, \pa_{\o\o}^2 u^n)  \le  \rho_G\left(\frac{2C+1}{n}\right).
\end{equation*}
\end{prop}
\begin{proof}
Let $(\a,\b,\g)\in \underline{\cJ}_L u^n(\th)$. Then for any $\e>0$ we have
\begin{equation}\label{2018-05-30:06}
  u^n(\th) > \ol\cE_L\Big[\big((u^n)^\th-\f^{\a+\e,\b,\g}\big)(\ch_\d, \B)\Big],
\end{equation}
for a suitably $\delta>0$ arbitrarily small.
By definition of $u^n$ and by \eqref{2018-05-30:06},
we can find $\th^*\in \Theta$ such that
\begin{equation*}
  \ola d_p(\th,\th^*)<\frac{2C+1}{n}
\q\mbox{and}\q
u(\th^*) - n \ola d_p(\th,\th^*) > \ol\cE_L\Big[\big((u^n)^\th-\f^{\a+\e,\b,\g}\big)(\ch_\d, \B)\Big].
\end{equation*}
Further, we have
\begin{equation*}
  \begin{split}
      u(\th^*) - n \ola d_p(\th,\th^*) &> \ol\cE_L\Big[\big((u^n)^\th-\f^{\a+\e,\b,\g}\big)(\ch_\d, \B)\Big]\\
&\ge   \ol\cE_L\Big[\big(u^{\th^*}-\f^{\a+\e,\b,\g}\big)(\ch_\d, \B) -n \ola d_p\big((t+\ch_\d, \o\otimes_t \B),(t^*+\ch_\d, \o^*\otimes_{t^*}\B)\big) \Big].
\end{split}
\end{equation*}
It is important to note that
\begin{equation*}
  \ola d_p(\th,\th^*) = \ola d_p\big((t+\ch_\d, \o\otimes_t \B),(t^*+\ch_\d, \o^*\otimes_{t^*}\B)\big).
\end{equation*}
Therefore,
\begin{equation*}
  u(\th^*) > \ol\cE_L \Big[\big(u^{\th^*}-\f^{\a+\e,\b,\g}\big)(\ch_\d, \B) \Big].
\end{equation*}
Now we apply Lemma \ref{lem:OS} and obtain that there exists $\tilde\th$ such that 
\begin{equation*}
   \tilde t<\ch_\d(\tilde\o)
 \q\mbox{and}\q
 (\a+\e,\b + \g\tilde\o_{\tilde t},\g) \in \underline{\cJ}_L u\big(t^*+\tilde t, \o^*\otimes_{t^*} \tilde\o\big),
\end{equation*}
and thus, by the subsolution property of $u$,
\begin{equation}
  \label{eq:ppdeestimate}
-\a -\e -G\big(t^*+\tilde t, \o^*\otimes_{t^*} \tilde\o,\b + \g\tilde\o_{\tilde t},\g \big) \le   0.
\end{equation}
Note that
\begin{equation}\label{2018-05-30:07}
  \begin{split}
      \ola d_p\big(\th, (t^*+\tilde t, \o^*\otimes_{t^*} \tilde\o) \big)
& \le 
\ola d_p\big((t+\tilde t, \o\otimes_t \tilde\o), (t^*+\tilde t, \o^*\otimes_{t^*} \tilde\o) \big)
+ \ola d_p\big(\th, (t+\tilde t, \o\otimes_{t} \tilde\o) \big)\\
& \le  \ola d_p(\th, \th^*) + 
\delta+
\left(  \d^{p+1} + T \big|\rho(\o_{t\we\cd},\d)\big|^p \right)^\frac1p
\\
& \le  \frac{2C+1}{n} + 
  \delta+
\left( \d^{p+1} + T \big|\rho(\o_{t\we\cd},\d)\big|^p \right)^\frac1p . 
\end{split}
\end{equation}
where $\rho(\o_{t\we\cd},\cd)$ is the modulus of continuity of the path $\omega_{t\wedge \cdot}$.
By using the definition of $\rho_G$
together with 
 \eqref{eq:ppdeestimate}
and
\eqref{2018-05-30:07},
we get
\begin{equation*}
  -\a -\e -G\left(\th,\b + \g\tilde\o_{\tilde t},\g \right) \le   \rho_G\left(\frac{2C+1}{n}+
\delta+
 \left( \d^{p+1} + T \right|\rho(\o_{t\we\cd},\d)\left|^p \right)^\frac1p \right).
\end{equation*}
Finally, 
we let $\d,\e$ tend to $0$
and obtain
$-\a -G(\th, \b,\g)  \le  \rho_G\left(\frac{2C+1}{n}\right)$.
\end{proof}

Now we are ready to prove the main result of this section.

\begin{thm}\label{2018-06-18:02}
Let $G$ satisfy Assumption~\ref{assum:Gcomp}. Let $u$ (resp.\ $v$) be a bounded $\ola d_p$-u.s.c.\ (resp.\ l.s.c.) $\cP_L$-viscosity subsolution  (resp.\ supersolution) to  PPDE \eqref{eq:PPDEgeneral}. In addition, assume that $u,v$ satisfy Assumption \ref{assum:terminalLip}. Then, if $u(T,\cdot) \le v(T,\cdot)$,  we have $u\le v$ on $\Th$.
\end{thm}
\begin{proof}
As a direct consequence of Proposition \ref{prop:unsubsolution}, we know that,
for $n$
sufficiently large,
\begin{equation*}
    \tilde u^n(\th) \coloneqq  u^n(\th) - \rho_G\left(\frac{2C+1}{n}\right)\left(T-t\right)
  \end{equation*}
  is a bounded $\ola d_p$-uniformly continuous $\cP_L$-viscosity subsolution to 
 PPDE \eqref{eq:PPDEgeneral}
and $\tilde u^n(T,\cdot) = u^n(T,\cdot)$. Further, by Lemma \ref{lem:unterminalconsistent}, we have $u^n(T,\cdot) =u(T,\cdot)$ and thus $\tilde u^n(T,\cdot) = u(T,\cdot)$ for $n$ big enough. We can similarly define $\tilde v^n$, so that 
$ \tilde v^n\rightarrow v$ pointwise
and that
$\tilde v^n$ is a bounded $\ola d_p$-uniformly continuous $\cP_L$-viscosity supersolution to \eqref{eq:PPDEgeneral} and $\tilde v^n(T,\cdot) =v(T,\cdot)$ for $n$ big enough. Since $\tilde v^n(T,\cdot) =v(T,\cdot)\ge u(T,\cdot) = \tilde u^n(T,\cdot)$, by Theorem~\ref{thm:comparisonuniform} we have
\begin{equation*}
  u = \lim_{n\rightarrow\infty} \tilde u^n \le \lim_{n\rightarrow\infty} \tilde v^n  = v,
\end{equation*}
and the proof is complete.
\end{proof}



\section{Representation 
of PPDEs as PDEs
in infinite dimension and comparison under weak-continuity}\label{sec:hilbert}

The aim of this section is to start with a $\mathcal{P}_L$-viscosity (sub-/super-)solution $u$ defined on the space $\Theta$, then to associate to it a function $\tilde{ u}$ defined on a product space $[0,T]\times H$ where $H$ is a suitably chosen Hilbert space, and finally show that $\tilde{u}$ is a viscosity (sub-/super-)solution of a PDE on $[0,T]\times H$.
As a corollary of such relationship, we can exploit the comparison theorem available for viscosity solutions in Hilbert spaces  to obtain uniqueness for $\mathcal{P}_L$-viscosity solutions.

We start by introducing the Hilbert space $H$ on which we will set our new PDE associated with the original PPDE
\eqref{eq:PPDEgeneral}.
Then we will address the problem of associating a function  on the original space $\Theta$ with a function on the product space $[0,T]\times H$.
In order to perform this change of variable in such a way to end up with a function regular enough to exploit the comparison theorem for viscosity solutions in Hilbert spaces, we need to  introduce  a pseudo-metric $d_B$ on $\Theta$, weaker than $\ola d_p$, and an associated norm $|\cdot|_B$ on $H$, weaker than the original norm $|\cdot|_H$.

Once provided these preliminaries,
we can introduce the PDE on $[0,T]\times H$ associated with the original PPDE on $\Theta$,
 recall the notion of viscosity solution in Hilbert spaces,
and prove the main theorem of this section
(Theorem~\ref{2017-01-31:05}),
thanks to which we can use 
\cite[Theorem~3.50]{Fabbri} to get uniqueness of $\mathcal{P}_L$-viscosity solutions (Corollary~\ref{2018-03-15:00}).

 For the theory of viscosity solutions in Hilbert spaces 
 we always refer to \cite{Fabbri}.
For the basic notions of stochastic calculus in Hilbert spaces
that we need,
 we refer to \cite{Peszat2007}.

\bigskip
We  start by introducing
 the Hilbert spaces
\begin{equation*}
H'\coloneqq
L^2(\mathbb{R}_-,\mathbb{R}^m),\qquad
  H\coloneqq \mathbb{R}^m\times H',
\end{equation*}
where $H'$ is endowed with its standard
scalar product
$\langle\cdot,\cdot\rangle_{L^2}$
induced by the $L^2$-norm
 and
$H$ is endowed with 
scalar product and norm given by
\begin{equation*}
  \langle x,x'\rangle_H\coloneqq \langle x_0,x'_0\rangle+\langle x_1,x'_1\rangle_{L^2},
\quad
  | x|_H\coloneqq
   \left( 
     |x_0|^2+| x_1|_{L^2}^2 \right) ^{1/2},
\quad \forall x=(x_0,x_1),\ x'=(x'_0,x'_1)\in \mathbb{R}^m\times
H'.
\end{equation*}
We next consider
the $C_0$-semigroup $S$ on $H$ defined by
\begin{equation}
  \label{eq:2017-02-13:03}
  S_t\colon H\rightarrow H,\ (x_0,x_1) \mapsto (x_0,x_0\mathbf{1}_{[-t,0]}+x_1(\cdot+t)\mathbf{1}_{(-\infty,-t)})\qquad \forall t\in\mathbb{R}^+.
\end{equation}
The
infinitesimal generator $A$ of $S$ is given by
\begin{equation}
  \label{eq:2017-02-13:04}
  A\colon D(A)\subset H\rightarrow H,\ (x_0,x_1) \mapsto (0,\dot x_1)
\end{equation}
where
\begin{equation}
  \label{eq:2017-02-13:05}
  D(A)\coloneqq \left\{(x_0,x_1)\in H\colon x_1\in W^{1,2}(\mathbb{R}^-,\mathbb{R}),\ x_0=x_1(0)\right\}.
\end{equation}

To express the regularity 
assumptions
for the comparison results (Theorem~\ref{2017-01-31:05}
and
Corollary~\ref{2018-03-15:00}),
we will
consider
 the following pseudo-metric $d_B$ on $\Theta$:
\begin{equation*}
  \begin{split}
      d_B(\theta,\theta'))\coloneqq
  &
  |t-t'|
  +|\omega(t)-\omega'(t')|+
 \left|
   \int_0^t \omega(r)dr
   -
   \int_0^{t'} \omega'(r)dr \right|
\\
   &+
   \left( 
     \int_0^T
      \left( 
        \int_{(t-\rho)\vee 0}^t
        \omega (r)dr
        -
        \int_{(t'-\rho)\vee 0}^{t'}
         \omega'(r)dr
      \right) ^2 d\rho
   \right) ^{1/2}
 \end{split}
\end{equation*}
for all 
$ \theta, \theta'\in \Theta$, and the following scalar product $\langle\cdot,\cdot\rangle_B$  and norm $|\cdot|_B$ on $H$:
\begin{equation*}
  \langle x,x'\rangle_B\coloneqq
 \langle (A-I)^{-1}x,(A-I)^{-1}x'\rangle_H\qquad
|x|_B\coloneqq |(A-I)^{-1}x|_H
\qquad
\forall x,x'\in H.
\end{equation*}
By a direct computation,
one can verify that
\begin{equation*}
  (A-I)^{-1}(x_0,x_1)=
  \left(-x_0,
    -e^\cdot x_0-\int_\cdot^0 e^{-(r-\cdot)}x_1(r)dr\right)\quad \forall x=(x_0,x_1)\in H.
\end{equation*}
Hence the norm $|\cdot|_B$ reads
\begin{equation}
\label{2017-02-15:00}
|x|_B
=
\left(
|x_0|^2+
\int_{-\infty}^0
 \left| 
e^sx_0
+\int_s^0
e^{s-r}x_1(r)dr
 \right| ^2
ds
\right)^{1/2}\quad \forall x\in H.
\end{equation}


By the very definition of $d_B$ and $|\cdot|_B$, we immediately have the following

\begin{lemma}\label{2018-03-12:00}
Let $\{\theta_n\}_n\subset \Theta$ be a sequence and let $\theta\in \Theta$.
Define $x^n_0\coloneqq \omega_n(t_n)$,
 $x^n_1\coloneqq \omega_n(\cdot+t_n)\mathbf{1}_{[-t_n,0]}$, $x_0\coloneqq \omega(t)$, $x_1\coloneqq \omega(\cdot+t)\mathbf{1}_{[-t,0]}$.
Then (with $x^n=(x_0^n,x_1^n)$, $x=(x_0,x_1)$)
\begin{equation*}
  t_n\rightarrow t\mbox{ and }
  |x^n-x|_B\rightarrow 0
  \quad
  \mbox{if and only if}
  \quad
  d_B(\theta_n,\theta)\rightarrow  0.
\end{equation*}
\end{lemma}

The following proposition provides an example of a functional on $\Theta$ which is $d_B$-continuous.
The proof is postponed to Appendix.

\begin{proposition}\label{2018-05-30:08}
Let $f\in L^2((0,T),\mathbb{R}^m)$.
Then the convolution
  \begin{equation}\label{2018-03-14:00}
    \Theta \rightarrow \mathbb{R},\
    \theta \mapsto 
    \int_0^t \langle f(r),\omega(t-r)\rangle dr
  \end{equation}
is $d_B$-continuous if and only if
$f\in W^{1,2}( (0,T),\mathbb{R}^m)$.
\end{proposition}

\smallskip
We now define the data for the PDE on $[0,T]\times H$ associated with the original PPDE~\eqref{eq:PPDEgeneral}.
Let the functions
\begin{equation*}
  G^{*_H},
  G_{*_H}
  \colon [0,T]\times H\times
  \mathbb{R}\times \mathbb{R}^m\times\mathbb{S}^m\rightarrow \mathbb{\overline R}
\end{equation*}
be associated with $G$ as follows:
for all
$(t,x,  r,\beta,\gamma)\in
(0,T)\times H\times\mathbb{R}\times \mathbb{R}^m\times \mathbb{S}^m$,
\begin{subequations}
  \begin{equation*}
  { G }^{*_H}(t,x,r,\beta,\gamma)\coloneqq
\limsup
_{
\substack{
\hat\theta\in \Theta,\
\hat{t}\rightarrow t\\
(\hat \omega(\hat t),\hat \omega(\cdot+\hat t)\mathbf{1}_{[-\hat t,0]})
\xrightarrow{|\cdot|_B}
(x_0,x_1\mathbf{1}_{[-t,0]})\\
r',\beta',\gamma'\rightarrow r,\beta,\gamma\\
}
}
G(\hat t,\hat \omega,\hat r,\hat \beta,\hat \gamma),
\end{equation*}
  \begin{equation*}
  { G }_{*_H}(t,x,r,\beta,\gamma)\coloneqq
\liminf
_{
\substack{
\hat\theta\in \Theta,\
\hat{t}\rightarrow t\\
(\hat \omega(\hat t),\hat \omega(\cdot+\hat t)\mathbf{1}_{[-\hat t,0]})
\xrightarrow{|\cdot|_B}
(x_0,x_1\mathbf{1}_{[-t,0]})\\
r',\beta',\gamma'\rightarrow r,\beta,\gamma\\
}
}
G(\hat t,\hat \omega,\hat r,\hat \beta,\hat \gamma).
\end{equation*}
\end{subequations}

\noindent Similarly, for $u\colon \Theta\rightarrow \mathbb{\overline R}$, define the functions $u^{*_H}, u_{*_H}\colon [0,T]\times H\rightarrow \mathbb{\overline R}$ by
\begin{subequations}
  \begin{equation*}
  u^{*_H}(t,x)\coloneqq
\limsup
_{
\substack{
\hat\theta\in \Theta,\
\hat{t}\rightarrow t\\
(\hat \omega(\hat t),\hat \omega(\cdot+\hat t)\mathbf{1}_{[-\hat t,0]})
\xrightarrow{|\cdot|_B}
(x_0,x_1\mathbf{1}_{[-t,0]})
}
}
u(\hat t,\hat \omega),
\end{equation*}
  \begin{equation*}
  u_{*_H}(t,x)\coloneqq
\liminf
_{
\substack{
\hat\theta\in \Theta,\
\hat{t}\rightarrow t\\
(\hat \omega(\hat t),\hat \omega(\cdot+\hat t)\mathbf{1}_{[-\hat t,0]})
\xrightarrow{|\cdot|_B}
(x_0,x_1\mathbf{1}_{[-t,0]})
}
}
u(\hat t,\hat \omega).
\end{equation*}
\end{subequations}

\noindent
It is clear
that  the functions ${ G}^{*_H},{ G}_{*_H},u^{*_H},u_{*_H}$ are well-defined, because 
for all $(t,x)\in [0,T]\times H$ we can find a sequence
$(t,\omega_n)\in \Theta$
such that $(\omega_n(t),\omega_n(\cdot+t)\mathbf{1}_{[-t,0]})$ converges to $(x_0,x_1\mathbf{1}_{[-t,0]})$ in the norm $|\cdot|_H$, hence in $|\cdot|_B$.

\bigskip

We will now recall the definition of viscosity solution for PDEs on Hilbert spaces
 as provided
by
\cite[Ch.\ 3]{Fabbri},
and it is in order to fit
such a framework that 
we write
\eqref{eq:2017-02-14:02} by emphasizing the maximal dissipative operator $A-I$.
Let $\hat G\colon [0,T]\times H\times \mathbb{R}\times \mathbb{R}^m\times \mathbb{S}^m\rightarrow \mathbb{\overline R}$ be a function.
Consider the following parabolic equation
\begin{equation}
  \label{eq:2017-02-14:02} 
  -v_t-\langle (A-I)x,D_xv\rangle-\langle x,D_xv\rangle-\hat G(t,x,v,D_{x_0}v,D^2_{x_0x_0}v)=0\qquad \mbox{on }(0,T)\times H,
\end{equation}
where $v\colon [0,T]\times H\rightarrow \mathbb{R}$ and where
$D_{x_0}v,D^2_{x_0x_0}v$ are the first and second order Fr\'echet differentials of $v$ with respect to the first component $x_0\in \mathbb{R}^m$ of the variable $x=(x_0,x_1)\in H=\mathbb{R}^m\times L^2((0,T),\mathbb{R}^m)$.
It is important to notice that
in 
 \cite[Ch.\ 3, Definitions~3.32]{Fabbri}
some assumptions are stated with respect 
to topologies induced by an operator denoted by $B$.
In our framework, we choose 
\begin{equation*}
  B\coloneqq (A^*-I)^{-1}(A-I)^{-1}.
\end{equation*}
Because of the compactness of $(A-I)^{-1}$ (Proposition~\ref{2017-02-15:02}), hence of $B$, 
and by
\cite[Lemma~3.6(i)]{Fabbri},
in our case the definitions of
test function  and
of  viscosity sub-/supersolution
given in \cite[Ch.\ 3, Definitions~3.32 and 3.35]{Fabbri} read as follows.

\begin{definition}[Test functions]
  A
 function $\psi\colon (0,T)\times H\rightarrow \mathbb{R}$ is a test function if $\psi(t,x)=\varphi(t,x)+h(t,|x|_H)$ , where:
  \begin{enumerate}[(i)]
  \item $\varphi\in C^{1,2}( (0,T)\times H, \mathbb{R})$, it is
locally bounded, weakly sequentially lower-semicontinuous,
$\nabla_x\varphi$ takes values in $D(A^*)$, and $ \partial _t\varphi,\ A^*\nabla_x\varphi,\ \nabla_x\varphi,\ D^2_{xx}\varphi$ are uniformly continuous on $(0,T)\times H$;
\item $h\in C^{1,2}((0,T)\times \mathbb{R},\mathbb{R})$ 
and is such that, for every $t\in (0,T)$, $h(t,\cdot)$ is even and $h(t,\cdot)$ is non-decreasing on $\mathbb{R}^+$.
  \end{enumerate}
\end{definition}


\begin{definition}[$H$-viscosity sub-/supersolution]\label{2018-02-21:00}
A
locally bounded
weakly sequentially u.s.c.\
  function $u\colon (0,T)\times H\rightarrow \mathbb{R}$
 is a $H$-viscosity subsolution of 
\eqref{eq:2017-02-14:02} 
if, whenever $u-\psi$ has a local maximum at a point $(t,x)\in (0,T)\times H$ for a test function $\psi(s,y)=\varphi(s,y)+h(s,|y|_H)$, then
\begin{equation}\label{eq:2017-02-21:00}
  -\psi_t(t,x)-\langle x,(A-I)^*\nabla_x\varphi(t,x)\rangle_H
  -\langle x,\nabla_x\psi(t,x)\rangle_H
  -\hat G(t,x,u(t,x),\nabla_{x_0}\psi(t,x),D^2_{x_0x_0}\psi(t,x))\leq 0.
\end{equation}
A
locally bounded
weakly sequentially l.s.c.\  function $v\colon (0,T)\times H\rightarrow \mathbb{R}$
 is a viscosity supersolution of 
\eqref{eq:2017-02-14:02} 
if, whenever $v+\psi$ has a local 
minimum at a point $(t,x)\in (0,T)\times H$ for a test function $\psi(s,y)=\varphi(s,y)+h(s,|y|_H)$, then
\begin{equation}
  \label{eq:2017-02-21:01}
  \psi_t(t,x)+\langle x,(A-I)^*\nabla_x\varphi(t,x)\rangle
  +\langle x,\nabla_x\psi(t,x)\rangle
  -\hat G(t,x,u(t,x),-\nabla_{x_0}\psi(t,x),-D^2_{x_0x_0}\psi(t,x))\geq 0.
\end{equation}
\end{definition}

\begin{remark}\label{2018-03-12:04}
  The Definition~\ref{2018-02-21:00} does not correspond exactly to \cite[Definition 3.35]{Fabbri}, because we drop the 
continuity assumption on  $\hat{ G}$.
We will recover such assumption 
 when dealing with  comparison.
\end{remark}


The first main result of this section is the following

\begin{theorem}\label{2017-01-31:05}
  Let $u$ be a
 $\mathcal{P}_L$-viscosity subsolution
(\emph{resp.}\ supersolution)
  of \eqref{eq:PPDEgeneral}.
Then  
${ u}^{*_H}$
(resp.~${u}_{*_H}$)
 is a $H$-viscosity subsolution
(\emph{resp.}\ supersolution)
   of
   \begin{subequations}
        \begin{equation}
     \label{2018-03-10:00}
       -v_t-\langle (A-I)x,D_xv\rangle-\langle x,D_xv\rangle- G^{*_H}(t,x,v,D_{x_0}v,D^2_{x_0x_0}v)=0
     \end{equation}
        \begin{equation}
     \label{2018-03-10:01}
\mbox{(resp.\ }       -v_t-\langle (A-I)x,D_xv\rangle-\langle x,D_xv\rangle- G_{*_H}(t,x,v,D_{x_0}v,D^2_{x_0x_0}v)=0
\mbox{).}
     \end{equation}
   \end{subequations}
 \end{theorem}

Theorem~\ref{2017-01-31:05}
allows to exploit the comparison for $H$-viscosity solutions (\cite[Theorem~3.50]{Fabbri}) to get
as corollary a comparison result for $\mathcal{P}_L$-viscosity solutions
(Corollary~\ref{2018-03-15:00}). As a byproduct, we also obtain a sufficient condition for the Perron-type result we proved in Section \ref{sec:perron}.

 \begin{corollary}\label{2018-03-15:00}
  Let $G$ satisfy Assumptions~\ref{2017-02-21:05}, with $d=d_B$.
  Let $u$ be a bounded $\cP_L$-viscosity subsolution of
  \eqref{eq:PPDEgeneral}  and
  let $v$ be a bounded $\cP_L$-viscosity supersolution of
  \eqref{eq:PPDEgeneral}.
  Suppose that
  \begin{equation}\label{eq:terminaluc}
    \lim_{r,\eta\rightarrow 0}
    \sup
     \left\{ 
       u(\th)-v(\th')\colon
      d_B(\th,\th')<r,\
       T-\eta\leq t,t'\leq T
     \right\} \leq 0.
  \end{equation}
Then
 $u\le v$ on $\Th$.
Moreover,
 Assumption \ref{2018-03-16:00}
is fulfilled  with $\underline u\coloneqq u,\overline v\coloneqq v, d\coloneqq d_B$.
 \end{corollary}

Before proving
Theorem~\ref{2017-01-31:05}
and 
Corollary~\ref{2018-03-15:00}, we need 
a preliminary 
discussion
useful to relate the two different settings of $\mathcal{P}_L$-viscosity solutions and of viscosity solutions on $H$ as defined by
Definition~\ref{2018-02-21:00}.

\smallskip
Let $U\colon [0,T]\times H\rightarrow \mathbb{R}$ be a
 measurable  function such that 
$U(t,x)=U(t,x_0,x_1\mathbf{1}_{[-t,0]})$ for $(t,x)\in [0,T]\times H$ ($x=(x_0,x_1)$).
Notice that,
for any $(s,x)\in H$,
 the measurability of $U$ entails the measurability of 
\begin{equation}
\label{2018-01-30:07}
  (\Theta,d_\infty)\rightarrow \mathbb{R},\  \theta \mapsto
U^{s,x}( \theta)\coloneqq
 U((s+t)\wedge T,x_0+\omega(s),x_1  \otimes_s[x_0, \omega]),
\end{equation}
where, for any function $\omega\in \Omega$ and $(s,x)\in [0,T]\times H$,
\begin{equation}\label{2017-12-19:02}
  x_1 \otimes _s[x_0, \omega]\coloneqq
  \begin{dcases}
    x_1(t+s) & \forall t\in (-\infty,-s)\\
    x_0+\omega(t+s) & \forall t\in [-s,0].
  \end{dcases}
\end{equation}

\noindent Let $(t,x)\in [0,T]\times H$
and let
$U\colon [0,T]\times H\rightarrow \mathbb{R}$ be a
measurable 
 function,  locally bounded from above (\emph{resp.}~locally bounded from below).
We define
the  jet 
$\underline{\mathcal{J }}_L^H U(t,x)$ 
(\emph{resp.}\ $\overline{\mathcal{J }}^H_L U(t,x)$)
by
\begin{multline*}
      \underline{\mathcal{J }}^H_L U(t,
x)\coloneqq
  \Big\{ 
    (\alpha,\beta,\gamma)\in \mathbb{R}\times \mathbb{R}^m\times \mathbb{S}^m\colon
    U(t,x)=
    \overline{\mathcal{E}}_L
     \left[ 
       (U^{t,x}-\varphi^{\alpha,\beta,\gamma})(\T
       \wedge 
       \ch_\delta
       ,\B)
     \right],\\
 \mbox{for some $\delta\in(0,T-t]$}
\Big\}
\end{multline*}
\begin{multline*}
\hskip-10pt\left(\mbox{\emph{resp.}}\  \overline{\mathcal{J }}^H_L U(t,x)\coloneqq
  \Big\{ 
    (\alpha,\beta,\gamma)\in \mathbb{R}\times \mathbb{R}^m\times \mathbb{S}^m\colon
    U(t,x)=
    \underline{\mathcal{E}}_L
     \left[ 
       (U^{t,x}-\varphi^{\alpha,\beta,\gamma})(\T\wedge 
       \ch_\delta
       ,\B)
     \right],\right.\\ 
\left.     \mbox{for some $\delta\in[0,T-t]$}
  \Big\} \right).
\end{multline*}

\medskip
Now let $u_n$ be a sequence of $d_\infty$-u.s.c.\ functions uniformly bounded from above and let us define a function $\overline{u}^H\colon [0,T]\times H\rightarrow \dbR$
(\emph{resp.\ }$\underline{u}_H$)
 by
 \begin{subequations}
   \begin{equation*}
\overline{u}^H(t,x) \coloneqq  
\limsup
_{
\substack{
\hat\theta\in \Theta,\
\hat{t}\rightarrow t\\
(\hat \omega(\hat t),\hat \omega(\cdot+\hat t)\mathbf{1}_{[-\hat t,0]})
\xrightarrow{|\cdot|_B}
(x_0,x_1\mathbf{1}_{[-t,0]})\\
n\rightarrow \infty
}
}
u_n(\hat t,\hat \omega)\quad \forall (t,x)\in[0,T]\times H,
\end{equation*}
   \begin{equation*}
\mbox{(\emph{resp}.\ }\underline{u}_H(t,x) \coloneqq  
\liminf
_{
\substack{
\hat\theta\in \Theta,\
\hat{t}\rightarrow t\\
(\hat \omega(\hat t),\hat \omega(\cdot+\hat t)\mathbf{1}_{[-\hat t,0]})
\xrightarrow{|\cdot|_B}
(x_0,x_1\mathbf{1}_{[-t,0]})\\
n\rightarrow \infty
}
}
u_n(\hat t,\hat \omega)\quad \forall (t,x)\in[0,T]\times H\mbox{)}.
\end{equation*}
\end{subequations}
\begin{prop}\label{prop:stable2}
Let $u_n,\overline u^H,\underline u^H$ be as above.
Then for any $(\a,\b,\g) \in \underline{\mathcal{J}}_L^H \overline{u}^H(t,x)$
(\emph{resp.\ }$(\a,\b,\g) \in
\overline{\mathcal{J}}_L^H
 \underline{u}_H(t,x)$), there exist $\th_n \in \Th$, $(\a_n,\b_n, \g)\in \ul\cJ_L u_n(\th_n)$ such that
 \begin{subequations}
   \begin{equation*}
 t_n\rightarrow  t,\quad 
(\omega_n(t_n),\omega_n(\cdot+t_n)\mathbf{1}_{[-t_n,0]})
\xrightarrow{|\cdot|_B}(x_0,x_1\mathbf{1}_{[-t,0]}),
\end{equation*}
\begin{equation*}
  \mbox{and }
 \big( u_n(\th_n),\a_n,\b_n\big) \rightarrow \big(\overline{u}^H(t,x), \a,\b\big)\quad
\mbox{(\emph{resp.\ }}
\big( u_n(\th_n),\a_n,\b_n\big) \rightarrow \big(\underline{u}_H(t,x), \a,\b\big)
\mbox{ ).}
\end{equation*}
\end{subequations}
\end{prop}
\begin{proof}
The desired result is very similar
to Proposition \ref{prop:stable}
and can be proved with  the same argument.
\end{proof}

\begin{proposition}\label{2018-01-30:08}
Let
$u$
be a 
 $\mathcal{P}_L$-viscosity subsolutions to~\eqref{eq:PPDEgeneral}.
Then
\begin{equation}
  \label{eq:2017-03-02:04}
  - \alpha
  -{G}^{*_H}(t,x,
  u^{*_H}(t,x),\alpha,\beta,\gamma)\leq 0,\qquad
\forall (t,x)\in (0,T)\times H,\
  \forall (\alpha,\beta,\gamma)\in \mathcal{\underline J}_L^H
 u^{*_H}(t,x).
\end{equation}
\end{proposition}
\begin{proof}
  Let $(t,x)\in (0,T)\times H$ and $(\alpha,\beta,\gamma)\in \underline{\mathcal{J}}_L^H u^{*_H}(t,x)$.
By the very definition of $ u^{*_H}$, we can apply 
Proposition~\ref{prop:stable2} to get sequences $ \theta_n\in \Theta$ and $(\alpha_n,\beta_n)\in \mathbb{R}\times \mathbb{R}^m$ such that 
$(\alpha_n,\beta_n,\gamma)\in \underline{\mathcal{J}}_Lu( \theta_n)$
and
\begin{equation}
  \label{2017-12-19:01}
t_n\rightarrow t,\quad
(\omega_n(t_n),
\omega_n(\cdot+ t_n)\mathbf{1}_{[- t_n,0]})
\xrightarrow{|\cdot|_B}(x_0,x_1\mathbf{1}_{[-t,0]}),
\quad 
( u_n(\th_n),\a_n,\b_n) 
\rightarrow (u^{*_H}(t,x), \a,\b).  \end{equation}
Since $u$ is a $\mathcal{P}_L$-viscosity solution to~\eqref{eq:PPDEgeneral}, we have, for all $n$,
$  - \alpha_n
  -G(t_n, \omega_n,
u_n(\theta_n),
\beta_n,\gamma)\leq 0$.
By taking the limit $n\rightarrow \infty$,
 by using
\eqref{2017-12-19:01},
and 
by the very definition of
 ${ G}^{*_H}$,
 we obtain
\begin{equation*}
    - \alpha
  - G^{*_H}(t,x,u^{*_H}(t,x),\alpha,\beta,\gamma)\leq 0,
\end{equation*}
which concludes the proof.
\end{proof}

In what follows, for $t\in[0,T]$ and $\mathbb{P}\in \mathcal{P}_L$, we denote by
$\mathbb{G}^t$  the translated filtration $\{\mathcal{G}_{0\vee (s-t)}\}_{s\in[0,T]}$ and by
 $\mathbb{\overline G}^{\mathbb{P},t}_+$ 
its
right-continuous $\mathbb{P}$-completion.
For $t\in[0,T]$, $x=(x_0,x_1)\in H$, 
we define the $H$-valued process
$\Z ^{t,x}=(\Z _0^{t,x},\Z ^{t,x}_1)$ as follows:
\begin{equation}\label{2018-01-28:01}
  \begin{split}
    &\mbox{ if }s\in [0,t],\quad
\Z ^{t,x}_{0,s}\coloneqq x_0\quad\mbox{and}\quad \Z ^{t,x}_{1,s}\coloneqq x_1\\
    &  \mbox{ if }s\in (t,T],\quad
    \Z ^{t,x}_{0,s}\coloneqq x_0+\A_{s-t}+\M_{s-t}
 \quad\mbox{and}\quad
    \Z ^{t,x}_{1,s}(r)\coloneqq
    \begin{dcases}
      x_1(r+s-t) &r\in(-\infty,-s+t) \\
      \Z ^{t,x}_{0,r+s}& r\in[-s+t,0).
    \end{dcases}
  \end{split}
\end{equation}
We also introduce the following functions
  \begin{equation*}
    \tilde{\mathbf{b}}\colon [0,T]\times \tilde \Theta\rightarrow H
    \qquad
    \sigmab
    \colon  \tilde \Theta\rightarrow L(\mathbb{R}^m,H)
  \end{equation*}
defined by
(\,\footnote{Recall that $\vartheta=((t,\omega),a,\mu,q)$.})
\begin{equation*}
\tilde{\mathbf{b}}_s(\vartheta)\coloneqq 
\begin{dcases}
(\dot a_s,0)  &\mbox{if }a\in 
W^{1,2}([0,T],\mathbb{R}^m)
\\
(0,0)&\mbox{otherwise}
\end{dcases}
\qquad\qquad
\mbox{and}
\qquad\qquad
\sigmab(\vartheta)(v)\coloneqq
(v,0)
\end{equation*}
for $\vartheta\in \tilde \Theta,s\in[0,T],v\in \mathbb{R}^m$.

\begin{lemma}\label{2017-12-19:04}
Let $\Z,\mathbf{\widetilde b}$ be as above.
Then
there exists an $\mathbb{R}^m$-valued predictable process $ \mathbf{b}$ such that,
for 
all $\mathbb{P}\in \mathcal{ P}_L$, 
$  \left\{ 
(s,\vartheta)\in [0,T] \times \tilde{\Theta}\colon 
\mathbf{b}_s(\vartheta)
\neq
\tilde{\mathbf{b}}_s(\vartheta)
 \right\} $
is contained in a $Lebesgue \otimes \mathbb{P}  $-null set and
\begin{equation}
  \label{eq:2017-02-13:08}
\mathbb{P}\mbox{-a.s.}\  \begin{dcases}
  \Z ^{t,x}_s= x& \forall s\in[0,t]\\
  \Z ^{t,x}_s= S_{s-t}x+\int_t^sS_{s-r}
\mathbf{b}_{r-t}dr
  +\int_t^sS_{s-r}\sigmab d\M^t_r
&\forall s\in(t,T],
  \end{dcases}
\end{equation}
where $\M^t\coloneqq\{\M_{0\vee (r-t)}\}_{r\in[0,T]}$ is a $(\mathbb{\overline G}^{\mathbb{P},t}_+,\mathbb{\overline P})$-martingale.
\end{lemma}
\begin{proof}
We  show the existence
of $\mathbf{b}$.
First, notice that the function
\begin{equation*}
  \rho\colon [0,T]\times \tilde{\Theta}\rightarrow L^2([0,T],\mathbb{R}^m),\ (s,\vartheta) \mapsto a_{s\wedge \cdot}
\end{equation*}
is predictable.
By using
\cite[p.\ 67, statement (2)]{Federer1969},
it is not difficult to show that
$W^{1,2}([0,T],\mathbb{R}^m)$ is a Borel subset of $L^2([0,T],\mathbb{R}^m)$.
Define $B\coloneqq
\rho^{-1}(W^{1,2}([0,T],\mathbb{R}^m))$.
It then follows that
the map
\begin{equation*}
  [0,T]\times \tilde{ \Theta}\rightarrow \mathbb{R},\ (s,\vartheta) \mapsto \mathbf{1}_{
B}(T,\vartheta)
\end{equation*}
is predictable, and hence also the map
\begin{equation*}
F\colon  [0,T]\times \tilde{\Theta}\rightarrow \mathbb{R}^m,\ (s,\vartheta) \mapsto \mathbf{1}_B(T,\vartheta)a_s
\end{equation*}
is predictable.
We then obtain the predictability of the left-hand side derivative
\begin{equation*}
  B_1\rightarrow \mathbb{R}^m,\ (s,\vartheta) \mapsto  \partial^- _sF(s,\vartheta),
\end{equation*}
on the predictable set
 $B_1$ of points $(s,\vartheta)$ where such derivative exists.
Since $\mathbb{P}\in \mathcal{P}_L$,
$Leb \otimes \mathbb{P} ([0,T]\times \tilde{\Theta}\setminus B_1)=0$,
hence $\mathbf{b}_s(\vartheta)
\coloneqq (\mathbf{1}_{B_1}(s,\vartheta) \partial ^-_sF(s,\vartheta),0)\in H$ satisfies the needed requirements.

Clearly 
\eqref{eq:2017-02-13:08} holds true if $s\in [0,t]$. We then assume $s\in (t,T]$.
Let $\mathbb{P}\in \mathcal{P}_L$.
For
 $y\in \mathbb{R}^m$,
by It\^o's formula
(\cite[Theorem D.2]{Peszat2007}) and noticing that
\begin{equation*}
  \langle(y_0,0),z\rangle_H=\langle(y_0,0),S_v z\rangle_H\quad \forall z\in H,\ v\geq 0,
\end{equation*}
 we have, $\mathbb{P}$-a.s.,
\begin{equation}\label{2018-03-08:00}
  \begin{split}
    \langle (y_0,0),\Z^{t,x}_s\rangle_H=&
 \langle y_0,\Z^{t,x}_{0,s}\rangle=
\langle y_0,x_0\rangle
+\int_t^s \langle y_0,\dot a_{r-t} \rangle dr
+
\int_t^s \langle y_0,d\M^t_r\rangle\\
=&
\langle (y_0,0),x\rangle_H
+\int_t^s \langle (y_0,0),S_{s-r}
\mathbf{b}_{r-t}
\rangle_H dr
+
\int_t^s \langle (y_0,0),S_{s-r}\sigma d\M^t_r\rangle_H\\
=&
\langle (y_0,0),x
+\int_t^s  S_{s-r}
\mathbf{b}_{r-t}
dr
+
\int_t^s S_{s-r}\sigma d\M^t_r
\rangle_H.
\end{split}
\end{equation}
Now let  $y_1\in L^2((-\infty,0),\mathbb{R}^m)$.
By
the very definition
of $\Z^{t,x}_s$,
we have, $\mathbb{P}$-a.s.,
\begin{equation}\label{2018-03-08:02}
  \begin{split}
    \langle (0,y_1),\Z^{t,x}_s\rangle_H
    =&
    \langle y_1,\Z^{t,x}_{1,s}\rangle_{L^2}
=
\int_{-\infty}^0
\langle
y _1(v),\Z^{t,x}_{1,s}(v)
\rangle
dv\\
    =&
    \int_{-\infty}^0
      \langle
      y_1(v),
     \left( 
       \mathbf{1}_{(-\infty,-s+t)}(v)
         x_1(v)
         +
                \mathbf{1}_{[-s+t,0)}(v)
\Z^{t,x}_{0,v+s} 
     \right)\rangle
dv\\
=&
\int_{-\infty}^0
\langle
y_1(v),
 \left( 
\mathbf{1}_{(-\infty,-s+t)}(v)
x_1(v)
+
\mathbf{1}_{[-s+t,0)}(v)
x_0 \right)
\rangle dv\\
&+
\int_{-\infty}^0
\langle
y_1(v),
\mathbf{1}_{[-s+t,0)}
(v)
 \left( 
\int_t^{v+s}
\dot
a_{r-t}dr
+
\int_t^{v+s}
d\M^t_r
 \right) \rangle dv\eqqcolon \mathbf{I}+\mathbf{II}
\end{split}
\end{equation}
Now observe that
\begin{equation}
  \label{2018-03-08:01}
  \mathbf{I}=
\langle
(0,y_1),
S_{s-t}x
\rangle_H.
\end{equation}
Moreover, by 
  using the stochastic Fubini theorem
(\cite[Theorem 8.14]{Peszat2007})
and denoting by $\M_i^t$ and $y^i_1$ the $i$th component of $\M^t$
and $y_1$, respectively,
 and by $e_i$ the $i$th element of the canonical basis of $\mathbb{R}^m$,
we have, $\mathbb{P}$-a.s.,
\begin{equation}
  \label{2018-03-08:03}
  \begin{split}
    \mathbf{II}=&
    \int_t^s
     \left( \int_{-\infty}^0
       \langle
       y_1(v),
       \mathbf{1}_{[r-s,0)}(v)
       \dot a_{r-t} 
\rangle
dv
        \right) 
       dr
+
\sum_{
i=1,\ldots,m}       
    \int_t^s
     \left( \int_{-\infty}^0
       y_1^i(v)
       \mathbf{1}_{[r-s,0)}(v)
dv
        \right) 
d\M_{i,r}^t
\\
       =&
       \int_t^s
                \langle
          (0,y_1),S_{s-r} \mathbf{b}_{r-t}
          \rangle_H dr
+ 
\sum_{
i=1,\ldots,m}       
   \int_t^s
\langle
(0,y_1),
S_{s-r}\sigma(e_i)
\rangle_H
d\M_{i,r}^t\\
=&
                \langle
          (0,y_1),     \int_t^s
S_{s-r} \mathbf{b}_{r-t} dr
          \rangle_H 
+ 
                \langle
          (0,y_1),      \int_t^s
S_{s-r}\sigma
d\M^t_r\rangle_H 
  \end{split}
\end{equation}
By
collecting
 \eqref{2018-03-08:00},
\eqref{2018-03-08:02},
\eqref{2018-03-08:01},
\eqref{2018-03-08:03}
and 
recalling that  $y_0,y_1$ were arbitrarily chosen, we 
obtain
\eqref{eq:2017-02-13:08}.
\end{proof}

%

Notice that the $H$-valued process $\Z^{t,x}$ defined by 
\eqref{2018-01-28:01}
is independent on $\mathbb{P}\in \mathcal{P}_L$ and everywhere continuous,
whereas equation
\eqref{eq:2017-02-13:08} depends on the chosen $\mathbb{P}$ through the chosen version of the stochastic integral.

\begin{proof}[\textbf{Proof of Theorem~\ref{2017-01-31:05}}]
We write the proof only for the subsolution case, since the other case is symmetric.

By the very definition,
${u}^{*_H}$ is $|\cdot|+|\cdot|_B$-u.s.c.,
hence by 
compactness of $(A-I)^{-1}$ (Proposition~\ref{2017-02-15:02})
 it is weakly sequentially u.s.c.
Now  let $\psi(s,y)=\varphi(s,y)+h(s,|y|_H)$ be a test function and let $(t,x)\in (0,T)\times H$ be a maximum for $ u^{*_H}-\psi$ on $[t-\xi,t+\xi]\times B(x;\xi)$, where $B(x;\xi)$ denotes the $(H,|\cdot|_H)$-ball centered in $x$ with radius $\xi\in (0,t\wedge (T-t))$.

Notice that $A-1/2$ is maximal dissipative.
Let $A_n\coloneqq nA(n-A)^{-1}$, $n\geq 1$, be the Yosida 
approximation of $A$.
Then there exists $\overline n\geq 1$ such that
$|S^n_t|_{L(H)}\leq e^t$, $t\geq 0$, $n\geq \overline n$,
where $S^n$ denotes the continuous semigroup generated by $A_n$.

For some $\epsilon\in (0,\xi\wedge (T-t))$ to be chosen later,
define the function
 $\tau_\epsilon\colon
 C([0,T],H)\rightarrow [t, T]$ by
\begin{equation*}
  \tau_\epsilon(f)\coloneqq \inf \left\{ 
    r\in[t,T]\colon 
    r-t
    +
    \sup_{v\in[t,r]}|f(v)-x|_H
    \geq \epsilon
     \right\}\qquad \forall f\in C([0,T],H).
\end{equation*}
Then
$\tau_\epsilon$ is continous.
Moreover, for any adapted $H$-valued continuous process 
$\Proc$, $\tau_\epsilon(\Proc)$ is a stopping time.

  Now  
\emph{we fix $\mathbb{P}\in \mathcal{P}_L$}.
For  $n\geq 1$, define the process
$\Z^{t,x}_n$ by
\begin{equation}
  \label{eq:2017-02-13:08n}
 \begin{dcases}
  \Z ^{t,x}_{n,s}\coloneqq x& \forall s\in[0,t]\\
  \Z ^{t,x}_{n,s}\coloneqq S^n_{s-t}x+\int_t^sS^n_{s-r}
\mathbf{b}_{r-t}dr
  +\int_t^sS^n_{s-r}\sigmab d\M_r^t
&\forall s\in(t,T].
  \end{dcases}
\end{equation}
We can choose for $\Z_n^{t,x}$ a 
continuous 
$(\mathbb{\overline G}^{\mathbb{P},t}_+,\mathbb{\overline P})$-version,
 and we do it (we refer to the discussion in \cite[Section~11.4]{Peszat2007}, after recalling that in our case the quadratic variation of $\M$ is bounded).
For any
$\mathbb{G}^t$-stopping time
$\rho\geq t$ 
and for any 
$f\in C([0,T],H)$,
denote
\begin{equation}
\label{2018-01-30:04}
  \tau(f)\coloneqq
\tau_{\rho,\epsilon}(f)
\coloneqq \rho\wedge \tau_\epsilon(f).
\end{equation}
Let
 $\Z$ be defined 
as in
\eqref{2018-01-28:01}.
Notice that $\tau(\Z^{t,x})$ is a 
$\mathbb{G}^t$-stopping time
and $\tau(\Z^{t,x}_n)$ is a 
$\mathbb{\overline G}^{\mathbb{P},t}$-stopping time.
For
$n\geq 1$, define
on 
 $(\tilde{\Theta},\mathbb{\overline G}^{\mathbb{P},t},\mathbb{\overline P})$
the
continuous processes $\Y^{t,x}$ and
$\Y^{t,x}_n$ by
\begin{equation}
  \label{2018-01-22:00}
 \begin{dcases}
  \Y ^{t,x}_s\coloneqq x& \forall s\in[0,t]\\
  \Y ^{t,x}_s\coloneqq S_{s-t}x+\int_t^s
 \mathbf{1}_{r<\tau(\Z^{t,x}) }S_{s-r}
\mathbf{b}_{r-t}dr
  +\int_t^s
\mathbf{1}_{r\leq \tau(\Z^{t,x}) }S_{s-r}\sigmab
 d\M^t_r
&\forall s\in(t,T],
  \end{dcases}
\end{equation}
and 
\begin{equation}
  \label{2018-01-22:00n}
 \begin{dcases}
  \Y ^{t,x}_{n,s}\coloneqq x& \forall s\in[0,t]\\
  \Y ^{t,x}_{n,s}\coloneqq S^n_{s-t}x+\int_t^s
 \mathbf{1}_{r\leq\tau(\Z^{t,x}_n) }S^n_{s-r}
\mathbf{b}_rdr
  +\int_t^s
\mathbf{1}_{r\leq\tau(\Z^{t,x}_n) }S^n_{s-r}\sigmab
 d\M_r^t
&\forall s\in(t,T].
  \end{dcases}
\end{equation}
Notice that,
by
Lemma~\ref{2017-12-19:04}
and by the very definition of $\Z_n^{t,x}$, 
we have, for all
$s\in (t,T]$,
 $n\geq  1$,
\begin{equation}
  \label{eq:2018-01-22:01a}
  \begin{dcases}
    \Y^{t,x}_s=\Z^{t,x}_s&  \mathbb{\overline P}\mbox{-a.s.\ on\ }\{\tau(\Z^{t,x})>s\}\\
    \Y^{t,x}_s=S_{s-\tau(\Z^{t,x})}\Z^{t,x}_{\tau(\Z^{t,x})}
    &  \mathbb{\overline P}\mbox{-a.s.\ on\ }\{\tau(\Z^{t,x})\leq s\}
  \end{dcases}
\end{equation}
\begin{equation}\label{eq:2018-01-22:01n}
  \begin{dcases}
    \Y^{t,x}_{n,s}=\Z^{t,x}_{n,s}&  \mathbb{\overline P}\mbox{-a.s.\ on\ }\{\tau(\Z^{t,x}_n)>s\}\\
    \Y^{t,x}_{n,s}=S^n_{s-\tau(\Z^{t,x}_n)}\Z^{t,x}_{n,\tau(\Z^{t,x}_n)}
    &  \mathbb{\overline P}\mbox{-a.s.\ on\ }\{\tau(\Z^{t,x}_n)\leq s\}.
  \end{dcases}
\end{equation}
Now, since $A_n$ 
generates
 a continuous group in $L(H)$, the process $\Y^{t,x}_n$ 
is
 the strong
solution to the following linear SDE in integral form on $(\tilde{ \Theta},\mathbb{\overline G}^{\mathbb{P},t},\mathbb{\overline P})$:
 \begin{equation}
   \label{2018-01-28:02}
   \begin{split}
        \Y^{t,x}_{n,s} =&x+\int_t^s  \left( 
     A_n \Y^{t,x}_{n,r} +
     \mathbf{1}_{r\leq\tau(\Z^{t,x}_n)}\mathbf{b}_{r-t} \right) dr
   +\int_t^s
     \mathbf{1}_{r\leq\tau(\Z^{t,x}_n)}
     \sigmab  d\M^t_r  \\
=&
x+
A_n
\int_t^s  \Y^{t,x}_{n,r} dr
+
\sigmab
 \left( \A_{\tau(\Z^{t,x}_n) \wedge s-t}
+\M_
{\tau(\Z^{t,x}_n) \wedge s-t} \right) 
\qquad\qquad\qquad \mathbb{P}\mbox{-a.s.},\ \forall s\in[t,T).   
\end{split}
\end{equation}
By \eqref{eq:2018-01-22:01n}
and \eqref{2018-01-28:02}, we have
 \begin{equation}\label{2018-01-29:05}
   \begin{split}
        \Z^{t,x}_{n,
          \tau(\Z^{t,x}_n)}
 =&x+\int_t^{\tau(\Z^{t,x}_n)
}  \left( 
     A_n \Z^{t,x}_{n,r} +
     \mathbf{b}_{r-t} \right) dr
   +\int_t^{\tau(\Z^{t,x}_n)}
     \sigmab  d\M_r^t  \\
=&
x+
A_n
\int_t^{\tau(\Z^{t,x}_n)}  \Z^{t,x}_{n,r} dr
+
\sigmab
 \left( \A_{\tau(\Z^{t,x}_n) -t}
+\M_
{\tau(\Z^{t,x}_n) -t} \right) 
  \qquad\mathbb{P}\mbox{-a.s.},\ \forall s\in[t,T).
\end{split}
\end{equation}

By
 the assumptions on the test function $\psi=\varphi+h$, by 
using 
\eqref{2018-01-29:05},
and by recalling
that
$|\Z^{t,x}_{n,\tau(\Z^{t,x}_n)\wedge \cdot}-x|_\infty\leq \epsilon$ for all $n$,
 we can apply 
It\^o's formula 
(\cite[Theorem D.2]{Peszat2007}) to $\psi(\tau(\Z^{t,x}_n),\Z^{t,x}_{n,\tau(\Z^{t,x}_n)})$ and take the expectation, to get,
 for  $n\geq \overline n$,
\begin{equation}\label{2018-01-29:02}
  \begin{split}
    \mathbb{E}^\mathbb{\overline P}
 \left[ 
  \psi(
\tau(\Z^{t,x}_n)
  ,
  \Z^{t,x}_{n,
\tau(\Z^{t,x}_n)
}
  )
 \right] =&  \psi(t,x)+
\mathbb{E}^\mathbb{\overline P}
 \left[ 
  \int_t^{\tau(\Z^{t,x}_n)}
    \left( 
      \partial_t \psi(r,\Z_{n,r}^{t,x})
      +
      \langle
      \nabla_x 
\psi(r,\Z_{n,r}^{t,x}),
      A_n\Z^{t,x}_{n,s}
      \rangle 
       \right)dr \right] \\
&+
\mathbb{E}^\mathbb{\overline P}
 \left[ 
  \int_t^{\tau(\Z^{t,x}_n)}
      \langle
      \nabla_x \psi(r,\Z_{n,r}^{t,x}),
      \mathbf{b}_{r-t}
      \rangle  dr
 \right] \\
&+      \frac{1}{2}\mathbb{E}^\mathbb{\overline P}
\left[
  \int_t^{\tau(\Z^{t,x}_n)} 
  \operatorname{Tr}
  \left[ 
    \sigmab
   \dot\Q_{r-t}
    \sigmab^*
    D^2_{xx}\psi
    (r,\Z^{t,x}_{n,r})
  \right] 
  dr \right]\\
\leq &
  \psi(t,x)+
\mathbb{E}^\mathbb{\overline P}
 \left[
  \int_t^{\tau(\Z^{t,x}_n)}  
    \left( 
      \partial_t \psi(r,\Z_{n,r}^{t,x})
      +
      \langle
      A_n^*\nabla_x \varphi(r,\Z_{n,r}^{t,x}),
      \Z^{t,x}_{n,r}
      \rangle
 \right) dr \right] \\ 
&+
\mathbb{E}^\mathbb{\overline P}
 \left[
  \int_t^{\tau(\Z^{t,x}_n)}  
      \langle
      \nabla_x h(r,\Z_{n,r}^{t,x}),
      \Z^{t,x}_{n,r}\rangle
       dr \right] \\
&+
\mathbb{E}^\mathbb{\overline P}
 \left[ 
  \int_t^{\tau(\Z^{t,x}_n)}  
      \langle
      \nabla_x \psi(r,\Z_{n,r}^{t,x}),
      \mathbf{b}_{r-t}
      \rangle  dr
 \right] \\
&+      \frac{1}{2}\mathbb{E}^\mathbb{\overline P}
\left[ 
  \int_t^{\tau(\Z^{t,x}_n)} 
  \operatorname{Tr}
  \left[ 
   \dot\Q_{r-t}
    D^2_{x_0x_0}\psi
    (r,\Z^{t,x}_{n,r})
  \right] 
  dr \right],
      \end{split}
\end{equation}
where we have used the fact that
 $h$ is radial and that
\begin{equation*}
  \langle A_ny,y\rangle\leq |y|^2_H\qquad\forall y\in H, \ n\geq \overline n.
\end{equation*}
We will now take in consideration each term appearing in the  formula above and pass to the limit as $n\rightarrow \infty$ to get a useful inequality for
 $\mathbb{E} \left[ \psi\left(
\tau(\Z^{t,x}_n),
\Z^{t,x}_{\tau(\Z^{t,x}_n)}\right)
\right] $.
Using
\eqref{eq:2017-02-13:08},
\eqref{eq:2017-02-13:08n},
and
 the
standard machinery
based on the
factorization formula for 
stochastic convolutions with $C_0$-semigroups 
(see \cite[Sections~11.3--4]{Peszat2007} and recall that the quadratic variation of $\M$ is bounded in our case)
one can  see that
\begin{equation}
  \label{eq:2018-01-28:03}
  \lim_{n\rightarrow \infty}
  \mathbb{E}^\mathbb{\overline P}
   \left[ 
     \left|
       \Z^{t,x}
       -
       \Z^{t,x}_n\right|^2_\infty
   \right] =0.
\end{equation}
This entails immediately the two following facts
\begin{subequations}
  \begin{equation}
    \label{eq:2018-01-30:00}
  \lim_{n\rightarrow \infty}
  \mathbb{E}^\mathbb{\overline P}
   \left[ 
     \left|
       \Z^{t,x}_{\tau(\Z^{t,x}_n)} 
       -
       \Z^{t,x}_{n,{\tau(\Z^{t,x}_n)} }\right|^2_\infty
   \right] =0,
  \end{equation}
  \begin{equation}
\label{2018-01-30:01}
\lim_{n\rightarrow \infty}  \tau_\epsilon(\Z^{t,x}_n)=
\  \tau_\epsilon(\Z^{t,x})
\qquad \mathbb{\overline P}\mbox{-a.s.}
 \end{equation}
\end{subequations}
By
\eqref{eq:2018-01-30:00} and
\eqref{2018-01-30:01} it follows that
\begin{equation}
  \label{eq:2018-01-30:02}
    \lim_{n\rightarrow \infty}
  \mathbb{E}^\mathbb{\overline P}
   \left[ 
     \left|
       \Z^{t,x}_{\tau(\Z^{t,x})} 
       -
       \Z^{t,x}_{n,{\tau(\Z^{t,x}_n)} }\right|^p_\infty
   \right] =0,
\end{equation}
for any $p\in[1,2)$.
By the assumptions on $\psi$
and by \eqref{eq:2018-01-30:02}
we can then 
pass to the limit in 
\eqref{2018-01-29:02} and obtain, 
 \begin{equation}
   \label{eq:2018-01-29:00}
   \begin{split}
     \mathbb{E}^\mathbb{\overline P} \left[ \psi\left(
\tau(\Z^{t,x}),\Z^{t,x}_{\tau(\Z^{t,x})}\right)
     \right] =& \lim_{n\rightarrow \infty}
     \mathbb{E}^\mathbb{\overline P} \left[ \psi\left(
\tau(\Z_n^{t,x}),\Z^{t,x}_{n,\tau(\Z^{t,x})}
\right)
     \right]\\
\leq& \psi(t,x)+ \mathbb{E}^\mathbb{\overline P} 
     \left[ 
 \left( 
       \int_t^{\tau(\Z^{t,x})}
       \partial_t \psi(r,\Z_{r}^{t,x}) 
     +\langle
       A^*\nabla_x \varphi(r,\Z_r^{t,x}), \Z^{t,x}_r \rangle \right) dr
     \right]\\
     &+ 
\mathbb{E}^\mathbb{\overline P} \left[ \int_t^{\tau(\Z^{t,x})} \left(   \langle \nabla_x h(r,\Z_r^{t,x}), \Z^{t,x}_r\rangle
 + 
   \langle \nabla_x
   \psi(r,\Z_r^{t,x}), \mathbf{b}_{r-t} \rangle
 \right)  dr
 \right] \\
 &+ \frac{1}{2}\mathbb{E}^\mathbb{\overline P} \left[ \int_t^{\tau(\Z^{t,x})}
   \operatorname{Tr}
   \left[  \dot\Q_{r-t} D^2_{x_0x_0}\psi (r,\Z^{t,x}_r)
   \right] dr \right].
\end{split}
\end{equation}
Since $\mathbb{\overline P}$ was arbitrary, 
\eqref{eq:2018-01-29:00} holds for all $\mathbb{\overline P}\in \mathcal{P}_L$.
Now, 
since 
$ \partial _t\psi$, $A^*\nabla_x\varphi$, $\nabla_x h$, $\nabla_x\psi$, $D^2_{x_0x_0}\psi$ are uniformly continuous on bounded sets,
recalling the definition of $\tau=\tau_{\rho,\epsilon}$ in
\eqref{2018-01-30:04},
and
noticing that
$\epsilon\downarrow 0$ implies $\tau_\epsilon(\Z^{t,x})\downarrow t$ $\mathbb{\overline P}$-a.s.,
one can easily see that
  \begin{multline*}
              \lim_{\epsilon\downarrow 0}
  \sup_{
\substack{
\mathbb{P}\in \mathcal{P}_L
\\\rho\geq t}}
 \mathbb{E}^\mathbb{\overline P} 
     \left[ 
\mathbf{1}_{ \tau (\Z^{t,x})>t}
  \left(\frac{1}{
 \tau (\Z^{t,x})-t}
       \int_t^{ \tau (\Z^{t,x})}
 \left( 
       \partial_t \psi(r,\Z_r^{t,x}) 
     +\langle
       A^*\nabla_x \varphi(r,\Z_r^{t,x}), \Z^{t,x}_r \rangle \right) dr
\right.\right.\\
\left.\left.\phantom{         \int_t^{ \tau (\Z^{t,x})}
}
 -      \partial_t \psi(t,x) 
     +\langle
       A^*\nabla_x \varphi(t,x), x \rangle
 \right)
 \right] =0
     \end{multline*}

     \begin{equation*}
       \lim_{\epsilon\downarrow 0} 
  \sup_{
\substack{
\mathbb{P}\in \mathcal{P}_L
\\\rho\geq t}}
       \mathbb{E}^\mathbb{\overline P} 
       \left[
         \mathbf{1}_{ \tau (\Z^{t,x})>t}
          \left( 
\frac{1}{ \tau (\Z^{t,x})-t}
\int_t^{\tau (\Z^{t,x})} \langle \nabla_x h(r,\Z_r^{t,x}),
         \Z^{t,x}_s\rangle dr 
         -\langle\nabla_x h(t,x),x\rangle
       \right) \right] =0
   \end{equation*}

\begin{equation*}
      \lim_{\epsilon\downarrow 0}
  \sup_{
\substack{
\mathbb{P}\in \mathcal{P}_L
\\\rho\geq t}}
 \mathbb{E}^\mathbb{\overline P} 
 \left[
 \frac{\mathbf{1}_{ \tau (\Z^{t,x})>t}
}{
 \tau (\Z^{t,x})-t}
 \left( 
      \int_t^{\tau(\Z^{t,x})}
   \langle \nabla_x
   \psi(r,\Z_r^{t,x}), \mathbf{b}_{r-t} \rangle
dr
- 
 \langle \nabla_x
   \psi(t,x), \A_{\tau(\Z^{t,x})-t} \rangle
 \right) 
 \right] =0
\end{equation*}

\begin{equation*}
      \lim_{\epsilon\downarrow 0}
  \sup_{
\substack{
\mathbb{P}\in \mathcal{P}_L
\\\rho\geq t}}
\mathbb{E}^\mathbb{\overline P} \left[
\frac{\mathbf{1}_{ \tau (\Z^{t,x})>t}}
{ \tau (\Z^{t,x})-t}
 \left( 
      \int_t^{ \tau (\Z^{t,x})}
 \operatorname{Tr}
   \left[  \dot\Q_r D^2_{x_0x_0}\psi (r,\Z^{t,x}_r)
   \right] dr 
-\operatorname{Tr}
 \left[ 
\Q_{ \tau(\Z^{t,x})-t}
D^2_{x_0x_0}\psi(t,x)
 \right]  \right) 
\right]=0.
\end{equation*}
Then, for any
arbitrarily small 
 real number
$\zeta>0$,
 there exists $\epsilon$ such that,
for all $ \mathbb{P}\in \mathcal{P}_L$
and  all 
$\mathbb{G}^t$-stopping time
$\rho\geq t$,
  \begin{subequations}
    \begin{equation}      \label{eq:2018-01-29:03a}
      \begin{multlined}[c][0.8\displaywidth]  
 \mathbb{E}^\mathbb{\overline P} 
     \left[ 
       \int_t^{ \tau (\Z^{t,x})}
 \left( 
       \partial_t \psi(r,\Z_r^{t,x}) 
     +\langle
       A^*\nabla_x \varphi(r,\Z_r^{t,x}), \Z^{t,x}_r \rangle \right) dr
     \right] \\
\leq
\mathbb{E}^\mathbb{P}
 \left[ 
 \left( 
       \partial_t \psi(t,x) 
     +\langle
       A^*\nabla_x \varphi(t,x), x \rangle  +\zeta \right)  
(\tau(\Z^{t,x})-t)
 \right] 
\end{multlined}
\end{equation}
\begin{equation}
     \label{eq:2018-01-29:03b}
      \begin{multlined}[c][0.8\displaywidth]  
 \mathbb{E}^\mathbb{\overline P} 
 \left[
       \int_t^{\tau (\Z^{t,x})}
  \langle \nabla_x h(r,\Z_r^{t,x}), \Z^{t,x}_s\rangle
   dr
 \right]
\leq
\mathbb{E}^\mathbb{P}
 \left[ 
 \left( \langle\nabla_x h(t,x),x\rangle
+\zeta \right) (\tau(\Z^{t,x})-t)
 \right] 
\end{multlined}
\end{equation}
\begin{equation}  \label{eq:2018-01-29:03c}
\begin{multlined}[c][0.9\displaywidth]   
\mathbb{E}^\mathbb{\overline P} 
 \left[
       \int_t^{ \tau(\Z^{t,x})}
   \langle \nabla_x
   \psi(r,\Z_r^{t,x}), \mathbf{b}_{r-t} \rangle
dr \right] 
\leq
\mathbb{E}^\mathbb{ P}
 \left[ 
 \langle \nabla_x
   \psi(t,x), \A_{\tau(\Z^{t,x})-t} \rangle
+\zeta
 \left( \tau(\Z^{t,x})-t  \right) \right] \\
=
\mathbb{E}^\mathbb{ P}
 \left[ 
 \langle \nabla_x
   \psi(t,x), \B_{\tau(\Z^{t,x})-t} \rangle
+\zeta
 \left( \tau(\Z^{t,x})-t  \right) \right]
\end{multlined}
\end{equation}
\begin{equation}
  \label{eq:2018-01-29:03d}
\begin{multlined}[c][0.9\displaywidth]   
\mathbb{E}^\mathbb{\overline P} \left[ \int_t^{
\tau(\Z^{t,x}) }
\operatorname{Tr}
   \left[  \dot\Q_r D^2_{x_0x_0}\psi (r,\Z^{t,x}_r)
   \right] dr 
 \right] \leq
\mathbb{E}^\mathbb{\overline P}
 \left[ \operatorname{Tr}
 \left[ 
\Q_{ \tau(\Z^{t,x})-t}
D^2_{x_0x_0}\psi(t,x)
 \right] 
+\zeta
 \left( 
\tau(\Z^{t,x})-t 
 \right)  \right] .
\end{multlined}
\end{equation}
\end{subequations}
By defining
\begin{equation*}
  \begin{split}
   \alpha\coloneqq 
       \partial_t \psi(t,x) 
     +\langle
       (A^*-I)\nabla_x \varphi(t,x), x \rangle  +
 \langle \nabla_x
   \psi(t,x), x\rangle+
4\zeta, \q \beta \coloneqq 
 \nabla_x\psi(t,x), \q \gamma\coloneqq D^2_{x_0x_0}\psi(t,x),
  \end{split}
\end{equation*}
 by
recalling that $(t,x)$ is a maximum for ${ u}^{*_H}-\psi$ on $[t-\xi,t+\xi]\times B(x;\xi)$,
and by
 collecting
\eqref{eq:2018-01-29:00},
\eqref{eq:2018-01-29:03a},
\eqref{eq:2018-01-29:03b},
\eqref{eq:2018-01-29:03c},
\eqref{eq:2018-01-29:03d},
we
 obtain,
for all $\mathbb{P}\in \mathcal{P}_L$
and all $\rho\geq t$,
\begin{equation}
  \label{eq:2018-01-29:04}
\begin{multlined}[c][0.85\displaywidth]    { u}^{*_H}(t,x)\geq
  \mathbb{E}^\mathbb{P}
  \left[ 
    { u}^{*_H}
    \left(\tau(\Z^{t,x}),
    \Z^{t,x}_{\tau(\Z^{t,x})}\right)
-
\alpha(\tau(\Z^{t,x})-t)
-\langle\beta,\B_{\tau(\Z^{t,x})-t}\rangle
-\frac{1}{2}\langle \gamma \M_{\tau(\Z^{t,x})-t},\M_{\tau(\Z^{t,x})-t}
\rangle
  \right] .
\end{multlined}
\end{equation}
We can finally conclude the proof.
By recalling the definition of the stopping time
$\ch_\delta$ in 
\eqref{2018-01-30:05} and 
by 
the very definition of
$\Z^{t,x}$
in \eqref{2018-01-28:01},
we see that  we can choose $\delta^*$ such that
\begin{equation*}
  t+\ch_{\delta^*}(\vartheta)\leq \tau_{\epsilon}(\Z^{t,x}(\vartheta))\qquad \forall \vartheta\in \tilde{\Theta}.
\end{equation*}
If we now define
$  \rho\coloneqq t+\T\wedge\ch_{\delta^*}$,
we get
$  \tau(\Z^{t,x})=
\rho\wedge \tau_{\epsilon}(\Z^{t,x})=
t+\T\wedge \ch_{\delta^*}$
and then
\eqref{eq:2018-01-29:04}
 provides
\begin{equation}
  \label{eq:2018-01-29:10}
u^{*_H}(t,x)=\mathcal{\overline E}_L
  \left[ 
     (u^{*_H})^{t,x}
    \left(
\T\wedge \ch _{\delta^*}
,\B
  \right)
-
\alpha
\T\wedge \ch_{\delta^*}
-\langle\beta,\B_{
\T\wedge \ch_{\delta^*}
}\rangle
-\frac{1}{2}\langle \gamma \M_{
\T\wedge \ch_{\delta^*}
},\M_{
\T\wedge \ch_{\delta^*}
}
\rangle
  \right],
\end{equation}
where we have used the equality
\begin{equation*}
  { u}^{*_H}
  \left( 
    t+\T\wedge \ch _{\delta^*}
    \Z^{t,x}_{    t+\T\wedge \ch _{\delta^*}
} \right) =
(u^{*_H})^{t,x}
 \left( 
\T\wedge \ch _{\delta^*},\B
 \right) 
  \end{equation*}
  (recall
\eqref{2018-01-30:07}
for the definition of $U^{s,y_0,y_1}$).
By \eqref{eq:2018-01-29:10} we have that
$  (\alpha,\beta,\gamma)\in \mathcal{\underline J}_L^Hu^{*_H}(t,x)$,
hence, by 
Proposition~\ref{2018-01-30:08},
we obtain
$  -\alpha-G^{*_H}(t,x,u^{*_H}(t,x),\alpha,\beta,\gamma)\leq 0$,
or, equivalently,
\begin{equation*}
-
        \partial_t \psi(t,x) 
     -\langle
       (A^*-I)\nabla_x \varphi(t,x), x \rangle -
 \langle \nabla_x
   \psi(t,x), x\rangle-
4\zeta
-
 G^{*_H}
 \left( t,x,
{u}^{*_H}(t,x),
 \nabla_x\psi(t,x),
 D^2_{x_0x_0}\psi(t,x) \right) \leq 0.
\end{equation*}
We can now conclude by letting $\zeta$
tends to $0$.
\end{proof}

\begin{proof}[\textbf{Proof of Corollary~\ref{2018-03-15:00}}]
First notice that,
by Theorem~\ref{2017-01-31:05},
$u^{*_H}$ and $v_{*_H}$
are  sub- and super-solutions of
        \begin{equation*}
       -u_t-\langle (A-I)x,D_xu\rangle-\langle x,D_xu\rangle- \hat G(t,x,u,D_{x_0}u,D^2_{x_0x_0}u)=0
     \end{equation*}
where
$\hat{ G}\coloneqq G^{*_H}=G_{*_H}$ is well-defined due to
Assumption~\ref{2017-02-21:05}\eqref{2018-03-16:01}.
We want to show that  
\cite[Theorem~3.50]{Fabbri}
applies to $u^{*_H},v_{*_H}$,
from which the corollary follows immediately.
But
it is easy to see that
Assumptions~\ref{2017-02-21:05}
imply that
Hypotheses~3.44--3.49 in \cite{Fabbri}
are verified 
with 
\begin{equation*}
  F(t,x,r,\beta,\gamma)\coloneqq -\langle x,\beta\rangle-\hat G(t,x,r,\beta,\gamma),
\end{equation*}
where $\hat{ G}\coloneqq G^{*_H}=G_{*_H}$.
In particular, 
the nuclearity condition Hypothesis~3.47 is automatically satisfied in our case because the second order term in $\hat G$ is finite dimensional, and
Hypothesis~3.48
comes from
the ellipticity condition
(Assumption~\ref{2017-02-21:05}\eqref{2018-03-16:02}),
from the local uniform continuity 
(Assumption~\ref{2017-02-21:05}\eqref{2018-03-16:01}),
and from the Lipschitz continuity of $G$
(Assumption~\ref{2017-02-21:05}\eqref{2018-03-20:05}).
Finally, 
condition (3.72) in
\cite[p.\ 206]{Fabbri},
with $u,v$ replaced by our $u^{*_H},v_{*_H}$,
follows by
\eqref{eq:terminaluc}.
\end{proof}


\appendix
\section{Appendix} \label{sec:preliminary}

We first
address some 
 properties of $\cP_L$.
Let $\gamma\in (0,1]$, $\ell > 0$.
Define
\begin{multline*}
  \tilde \Theta^{\gamma,\ell}_L\coloneqq  \left\{ \vartheta\in  \tilde \Theta\colon 
t\in [0,T],\ \omega = a +\mu,\right.\\
| a _s- a _{s'}|\leq L|s-s'|,\
\left.|\mu_s-\mu_{s'}|\leq \ell |s-s'|^\gamma,\ 
| q _s- q _{s'}|\leq L|s-s'|,\ 
\forall s,s'\in[0,T] 
 \right\}.
\end{multline*}

\begin{remark}
 \label{2016-09-27:00}
By applying Ascoli-Arzel\`a theorem, 
we
 see that $\A(\tilde \Theta^{\gamma,\ell}_L)$
and
$\M(\tilde \Theta^{\gamma,\ell}_L)$
 are relatively compact in $\Omega$, and
 $\Q(\tilde \Theta^{\gamma,\ell}_L)$ is relatively compact in 
$C_0([0,1],\mathbb{S}^m)$.
 Hence also $\B(\tilde \Theta^{\gamma,\ell}_L)$ is relatively compact.
Since $\tilde \Theta^{\gamma,\ell}_L$ is closed
 in $\tilde \Theta$,
we conclude that $\tilde \Theta^{\gamma,\ell}_L$ is 
compact in $\tilde \Theta$.
\end{remark}

\begin{lemma}\label{2016-09-26:06}
For all $\gamma\in \big(0,\frac12\big)$, and $\epsilon>0$, there exists $\ell>0$ such that
  \begin{equation}
    \label{eq:2016-09-23:02}
   \inf_{\mathbb{P}\in \mathcal{ P}_L} \mathbb{P}(\tilde \Theta^{\gamma,\ell}_L)\geq 1-\epsilon.
  \end{equation}
\end{lemma}
\begin{proof}
Notice that, if  $ a \in \A(\tilde \Theta_L)$
 and
 $ q \in \Q(\tilde \Theta_L)$,
where $\tilde\Theta_L$ is defined in 
\eqref{2016-09-27:01},
then 
$| a _s- a _{s'}|\leq L |s-s'|$
and
$| q _s- q _{s'}|\leq L |s-s'|$, for all $s,s'\in [0,T]$.
Moreover, since by definition $\mathbb{P}(\tilde \Theta_L)=1$ for all $\mathbb{P}\in \mathcal{ P}_L$,
we only need to show that there exists $\ell >0$ such that
\begin{equation}\label{2016-09-26:04}
  \inf_{\mathbb{P}\in \mathcal{P}_L}\mathbb{P}(
\overline \Theta^{\g, \ell})\geq 1-\epsilon,
\end{equation}
where 
$  \overline \Theta^{\g, \ell} \coloneqq  \left\{ 
\vartheta\in \tilde \Theta\colon\ 
|\mu_s-\mu_{s'}|\leq \ell |s-s'|^\gamma,\ \forall s,s'\in[0,T] 
 \right\} $. 
But now
 \eqref{2016-09-26:04} 
follows by
the Kolmogorov-\v Centsov continuity theorem (\cite[Ch.\ 2, Theorem 2.8]{Karatzas1991}), after
observing by inspection
 that the proof of the theorem holds uniformly in the reference probability.
\end{proof}

\ms

\begin{proof}[\textbf{Proof of Proposition \ref{2017-07-18:26}}]
 By recalling that
 $\tilde \Theta^{\gamma,\ell}_L$ is compact (Remark~\ref{2016-09-27:00}), and
by applying Prokhorov theorem  taking into account
 Lemma~\ref{2016-09-26:06}, we obtain that $\mathcal{P}_L$ is tight.
To conclude the proof, it remains to show that 
$\mathcal{P}_L$ is closed.
Let $\{\mathbb{P}_n\}_{n\in \mathbb{N}}\subset \mathcal{P}_L$ be a sequence converging to $\mathbb{P}$ in $\mathcal{P}$.
We need to show that 
$\mathbb{P}\in \mathcal{P}_L$, i.e.,
\begin{equation*}
  \mathbb{P}(\tilde \Theta_L)=1,\qquad \M\mbox{ is a $\mathbb{P}$-martingale},\qquad \langle \M\rangle=\Q\quad \mathbb{P}\mbox{-a.s.}
\end{equation*}

\no\underline{\emph{Step 1.}} We show that $\mathbb{P}(\tilde \Theta_L)=1$.
We first notice that $\tilde \Theta_L$ is closed.
%
So by weak convergence of $\{\mathbb{P}_n\}_{n\in \mathbb{N}}$ to $\mathbb{P}$, we have
\begin{equation*}
  \mathbb{P}(\tilde \Theta_L)\geq \limsup_{n\rightarrow \infty}\mathbb{P}_n(\tilde \Theta_L)=1.
\end{equation*}


\ms
\no\underline{\emph{Step 2.}} 
We show that $\M$ is a $\mathbb{P}$-martingale.
First, notice that the process $|\M|_\infty$ is  continuous and that 
(for some constants $c_1>0$)
\begin{equation*}
  \mathbb{E}^{\mathbb{P}_n}|\M|_\infty\leq c_1\mathbb{E}^{\mathbb{P}_n}
 \left[ 
 \left( \int_0^T
|\dot\Q_r| dr
 \right) ^{1/2}
 \right] \leq  c_1T^{1/2}L^{1/2} \qquad \mbox{for all}\q n\in \mathbb{N}.
\end{equation*}
Hence, by weak convergence, we have
\begin{equation}\label{2016-09-28:07}
  \mathbb{E}^{\mathbb{P}}|\M|_\infty=
\lim_{j\rightarrow  \infty}
  \mathbb{E}^{\mathbb{P}}[j\wedge |\M|_\infty]= 
\lim_{j\rightarrow  \infty}
\lim_{n\rightarrow  \infty}
  \mathbb{E}^{\mathbb{P}_n}[j\wedge | \M |_\infty]\leq
c_1
T^{1/2}L^{1/2}.
\end{equation}
Let  now $\tau_k(\vartheta)\coloneqq\inf\{s\in[0,T]\colon s+|\mu_{s\wedge \cdot}|_\infty\geq k\}\wedge t$. 
Notice that $\tau_k$ is continuous on $\tilde \Theta$
and that it is an $\mathbb{G}$-stopping time.
Hence
$ \M _{\tau_k}$ is continuous and bounded on $\tilde \Theta$.
Since $ \M $
 is a $\mathbb{P}_n$-martingale for all $n$, we can write, 
for all $t,t'\in[0,T]$, $t'<t$,
$\varphi$ bounded continuous $\cG_{t'}$-measurable functions,
 using the weak convergence of $\{\mathbb{P}_n\}_n$,
\begin{equation}
\label{2017-02-13:02}
  \mathbb{E}^{\mathbb{P}} \left[ \varphi  \M _{\tau_k\wedge t'} \right] =
\lim_{n\rightarrow \infty}  \mathbb{E}^{\mathbb{P}_n} \left[ \varphi  \M _{\tau_k\wedge t'} \right] 
=
\lim_{n\rightarrow \infty}  \mathbb{E}^{\mathbb{P}_n} \left[ \varphi  \M _{\tau_k\wedge t} \right] 
=
  \mathbb{E}^{\mathbb{P}} \left[ \varphi  \M _{\tau_k\wedge t} \right].
\end{equation}
This shows that
 $ \M _{\tau_k\wedge \cdot}$ is a $\mathbb{P}$-martingale.
Due to the fact that $\tau_k
\nearrow
 t$ pointwise as $k \rightarrow\infty$,
we obtain, 
by passing to the limit
in
\eqref{2017-02-13:02}
 after recalling
\eqref{2016-09-28:07},
\begin{equation*}
  \mathbb{E}^{\mathbb{P}} \left[ \varphi  \M _{t'} \right] =
  \mathbb{E}^{\mathbb{P}} \left[ \varphi  \M _t \right],
\end{equation*}
and
hence
we conclude that
 $ \M $ is a $\mathbb{P}$-martingale.
\ms

\no\underline{\emph{Step 3.}} We show that $\langle  \M \rangle= \Q $ $\mathbb{P}$-a.s.
To this aim, it is sufficient to show that $  \M ^{(i)} \M ^{(j)}- \Q ^{(i,j)}$ is a $\mathbb{P}$-martingale, where $ \M ^{(i)}$ denotes the $i$th component of the vector $ \M $ and $ \Q ^{(i,j)}$ denotes the $(i,j)$ entry of the matrix $ \Q $.
After noting that
that (for some $c_2>0$)
\begin{equation*}
  \mathbb{E}^{\mathbb{P}_n}\left[| \M^{(i)}\M^{(j)}- \Q^{(i,j)} |_\infty\right]\leq
c_2\mathbb{E}^{\mathbb{P}_n} \left[ \int_0^T    |\dot\Q| _r dr \right] 
\leq c_2 T L \qquad \forall n\in \mathbb{N},
\end{equation*}
we can proceed exactly as in \emph{Step 2}.
\end{proof}

\ms

In the following discussion, we aim to prove the dynamic programming result stated in Proposition \ref{prop:dpp}.


\ms

\begin{proof}[\textbf{Proof of Lemma \ref{lem:graphclose}}]
We need to show that
\begin{equation*}
\dbP^*\in \cP_L\q
\mbox{and}\q
\T \le T-\t(\vartheta^*),\ \dbP^*\mbox{-a.s.}
\end{equation*}
The former is due to the result of Proposition \ref{2017-07-18:26}, so it remains to prove the latter. Since $\dbP^n\rightarrow\dbP^*$, by Skorohod representation there exists a probability space $(\widehat \O, \widehat \cG, \widehat \dbP)$ on which
\begin{equation}\label{SkorohodR}
\exists\q \mbox{r.v.'s}\q \X^n \Big|_{\widehat \dbP} \stackrel{d}{=} \X \Big|_{\dbP^n},\q
\X^* \Big|_{\widehat \dbP} \stackrel{d}{=}\X \Big|_{\dbP^*}\q \mbox{such that}\q \big|\X^n- \X^*\big|_\infty\rightarrow 0, \q \widehat \dbP\mbox{-a.s.},
\end{equation}
where $ \stackrel{d}{=} $ denotes equality in distribution.
Since $\dbP^n\in  \cP_L (\t,\vartheta^n)$ and $\t$ is continuous, it follows from Fatou's lemma that
\begin{equation*}
1  =  \limsup_{n\rightarrow\infty} \widehat\dbP\Big[ \T^n \le T - \tau(\vartheta^n)\Big]
 \le    \widehat\dbP\Big[ \T \le T - \tau(\vartheta^*)\Big]
 =  \dbP^*\Big[ \T \le T - \tau(\vartheta^*)\Big],
\end{equation*}
which concludes
 the proof.
\end{proof}

\begin{prop}\label{prop:mea-sele}
Let $f:\Th \rightarrow \dbR$ be $d_\infty$-u.s.c.\ and bounded from above,  and $\t$ be a continuous $\dbG$-stopping time. Define
\begin{equation}
  \label{eq:valuefun1}
V(\vartheta) \coloneqq  \overline{\mathcal{E}}^\t_L \big[f(\T^{\t,\vartheta}, \B^{\t,\vartheta})\big](\vartheta)
= \overline{\mathcal{E}}^\t_L \big[f^{t\we \t(\vartheta), \o}\big](\vartheta).
\end{equation}
Then, there is a $\cG_{\T\we\t}$-measurable
 kernel $\nu: \tilde \Th\rightarrow \cP $ such that $\nu(\vartheta)\in  \cP_L (\t,\vartheta)$ for all $\vartheta\in \tilde\Th$ and
\begin{equation}\label{2018-05-30:11}
V(\vartheta) = \dbE^{\nu(\vartheta)}\big[f^{t\we\t(\vartheta),\o}\big].
\end{equation}
\end{prop}
\begin{proof}
First note that the equality in \eqref{eq:valuefun1} follows from Remark \ref{rem:2shifts}.
Consider the set $\mathcal{D}\subset \widetilde \Theta\times \mathcal{P}_L$ defined by
$  \mathcal{D}\coloneqq
  \left\{(\vartheta,\mathbb{P})\colon \vartheta\in \widetilde \Theta,\ \mathbb{P}\in \mathcal{P}_L(\tau,\vartheta)
  \right\}$.
For all $\dbP\in \cP_L(\t,\vartheta)$ define $\overline f\colon \mathcal{D}\rightarrow \mathbb{\overline R}$ by
\begin{equation*}
\ol f(\vartheta, \dbP) \coloneqq  \dbE^\dbP\big[f^{t\we\t(\vartheta),\o}\big].
\end{equation*}
We claim that $\ol f$ is jointly u.s.c.,
where on $\mathcal{D}$ we consider the topology induced by $\widetilde \Theta\times \mathcal{P}_L$, with $\widetilde \Theta$ endowed with its product topology.
Then let $|\vartheta^n-\vartheta^*|_\infty\rightarrow  0$ and $\dbP^n\rightarrow \dbP^*$. By Skorohod's representation, we have a probability space $(\widehat\O, \widehat\cG,\widehat\dbP)$ in which \eqref{SkorohodR} holds true. It follows that
\begin{equation*}
  \begin{split}
    \limsup_{n\rightarrow\infty}\ol f(\vartheta^n,\dbP^n)
    &= \limsup_{n\rightarrow\infty} \dbE^{\widehat\dbP}\Big[f\big(t^n\we\t(\vartheta^n)+ \T^n, \o^n\otimes_{t^n\we\t(\vartheta^n)} \B^n \big)\Big]\\
     &\le  \dbE^{\widehat\dbP}\Big[\limsup_{n\rightarrow\infty} f\big(t^n\we \t(\vartheta^n)+\T^n, \o^n\otimes_{t^n\we\t(\vartheta^n)} \B^n \big)\Big]\\
     & \le \dbE^{\widehat\dbP}\Big[f\big(t^*\we\t(\vartheta^*)+\T^*,
     \o^*\otimes_{t^*\we\t(\vartheta^*)} \B^* \big)\Big] = \ol
     f(\vartheta^*,\dbP^*),
  \end{split}
\end{equation*}
where the first inequality is due to Fatou's lemma and the second one is due to the $d_\infty$-u.s.c.\ of $f$.

Next note that
\begin{equation*}
V(\vartheta) = \sup_{\dbP\in \cP_L(\t,\vartheta)} \ol f(\vartheta,\dbP).
\end{equation*}
Taking into account that,
by Lemma \ref{lem:graphclose},
 $\mathcal{D}$ is closed,
and considering the
upper semicontinuity of $\ol f$, we can
then apply 
Proposition~7.33, p.\ 153 in \cite{BS}, to get
a 
Borel-measurable
kernel $\overline\nu\colon \widetilde\Theta\rightarrow \mathcal{P}$ such that $\overline\nu(\vartheta)\in \mathcal{P}_L(\tau,\vartheta)$ for all $\vartheta\in \widetilde \Theta$ and such that \eqref{2018-05-30:11} holds true if $\nu$ is replaced by $\overline \nu$.
Now we define
$\nu(\vartheta)\coloneqq \overline \nu(\vartheta_{t\wedge \tau(\vartheta)\wedge \cdot})$, for all $\vartheta\in \widetilde \Theta$.
Notice that the map
$\vartheta \mapsto 
\vartheta_{t\wedge \tau(\vartheta)\wedge \cdot}$
is measurable
from 
$\widetilde \Theta$
endowed with the sigma-algebra $\mathcal{G}_{\T\wedge \tau}$ 
into 
$\widetilde \Theta$
endowed with the Borel sigma-algebra $\mathcal{G}_T$.
Hence the Borel measurability of $\nu$ entails the $\mathcal{G}_{\T\wedge \tau}$-measurability of 
$\nu$.
Moreover,  we have $V(\vartheta)=V(\vartheta_{t\wedge \tau(\theta)\wedge \cdot})$
by the very definition of $V$
and $\rho(\vartheta)=\rho(\vartheta_{\rho(\vartheta)\wedge \cdot})$ for any stopping time $\rho$.
It follows
\begin{equation*}
  V(\vartheta)=
  V(\vartheta_{t\wedge \tau(\vartheta)\wedge \cdot})
  =\mathbb{E}^{\overline\nu(\vartheta_{t\wedge \tau(\vartheta)\wedge \cdot})}
   \left[ f^{t\wedge\tau(\vartheta_{t\wedge \tau(\vartheta)\wedge \cdot}),\omega_{t\wedge \tau(\vartheta)\wedge\cdot}} \right] =
\mathbb{E}^{\nu(\vartheta)}
   \left[ f^{t\wedge\tau(\vartheta),\omega} \right],
\end{equation*}
which concludes the proof.
\end{proof}

Let $\dbP\in \cP_L$, $\tau$ be a continuous stopping time, and $\nu\colon \tilde \Th\rightarrow \cP $ be a $\cG_{\T\we\t}$-measurable kernel such that $\nu(\vartheta)\in  \cP_L(\t,\vartheta)$ for $\dbP$-a.e.\ $\vartheta\in \tilde\Th$, and define the concatenation:
\begin{equation}
  \label{def:concatenation-prob}
\dbP\otimes_\t \nu(A) = \int\int (1_A)^{\t,\vartheta}(\vartheta')\nu(d\vartheta';\vartheta)\dbP(d\vartheta),\q A\in \cG_T.
\end{equation}

\begin{lem}\label{lem:concatenation-exp}
 Given a
measurable
 function $f\colon\Th\rightarrow R$ bounded from above, we have
 \begin{equation*}
   \dbE^\dbP \Big[
\dbE^{\nu(\X)} 
\big[ f^{\theta} \big]_{|\theta=(
\T\we \t, \B)
} \Big] =\dbE^{\dbP\otimes_\t \nu}[f].
\end{equation*}
Moreover, we have $\dbP\otimes_\t \nu \in \cP_L$.
\end{lem}
\begin{proof}
The first statement follows from the definition 
\eqref{def:concatenation-prob} if $f$ is the indicator function $\mathbf{1}_A$.
The general case follows 
by a standard approximation procedure.

Now the fact
$\mathcal{P} \otimes _\tau\nu\in \mathcal{P}_L$
follows easily by considering
that $\mathbb{P}\in \mathcal{P}_L$, that 
$\nu$ takes values in $\mathcal{P}_L$, and by using
\eqref{def:concatenation-prob}.
\end{proof}

Finally, we prove the dynamic programming result.

\ms
\begin{proof}[\textbf{Proof of Proposition \ref{prop:dpp}}]
By definition of nonlinear expectation and by the fact that any $\mathcal{P}_L(\tau,\vartheta)$ contains the dirac measure in $\vartheta=0$, we immediately have the inequality
\begin{equation*}
\ol\cE_L\Big[ f 1_{\{\T<\t\}}+\ol\cE^\t_L \big[f^{\theta}\big]
_{|
\theta=(\t,\B)}1_{\{\t\le\T\}
}
 \Big] 
\geq 
 \ol\cE_L\big[f\big].
\end{equation*}
Regarding the converse inequality, we first note that, by Proposition~\ref{prop:mea-sele}, there is a $\cG_{t\we\t}$-measurable kernel $\nu\colon\tilde \Th\rightarrow \cP $ such that
\begin{equation*}
\overline{\mathcal{E}}^\t_L \big[f^{t\we \t,\o}\big](\vartheta) =  \dbE^{\nu(\vartheta)}\big[f^{t\we\t,\o}\big].
\end{equation*}
Therefore
\begin{equation}
  \label{eq:19102017-1}
\dbE^\dbP\Big[\ol\cE^\t_L 
\big[f^{\theta}\big]_{|\theta=(\t,\B)} 1_{\{\t\le\T\}}\Big]
= \dbE^\dbP\Big[\ol\cE^\t_L \big[
f^{\theta}\big]_{|\theta=(\T\we\t,\B)} 1_{\{\t\le\T\}}\Big]
= \dbE^\dbP \Big[\dbE^{\nu(\X)} 
\big[ f^{\theta} \big]_{|\theta=(\T\we \t, \B)} 1_{\{\t\le\T\}}\Big] .
\end{equation}
Now observe that if any function $g\colon \tilde\Theta\rightarrow \mathbb{R}$ is $\mathcal{G}_{\T\wedge \tau}$-measurable, then $g^{\T\wedge\tau,\vartheta} (\cdot)$ is constantly equal to $g(\vartheta_{\T\wedge \tau\wedge \cdot})$.
By applying this fact to the function $g=\mathbf{1}_{\tau\leq \T}(\T,\B)$ and recalling
Remark~\ref{rem:2shifts},
it follows from Lemma~\ref{lem:concatenation-exp}
 that for any $\dbP\in \cP_L$
\begin{equation}
  \label{eq:19102017-2}
\dbE^\dbP \Big[\dbE^{\nu(\X)} \big[ 
f^{\theta} \big]_{\theta=(\T\we \t, \B)} 1_{\{\t\le\T\}}\Big] = \dbE^{\dbP\otimes_\t \nu}\Big[f 1_{\{\t\le\T\}}\Big].
\end{equation}
By applying the same observation  to the $\mathcal{G}_{\T\wedge \tau}$-measurable function $f(\T,\B)\mathbf{1}_{\T<\tau}(\T,\B)$ and by
combining \eqref{eq:19102017-1} and \eqref{eq:19102017-2} with Lemma~\ref{lem:concatenation-exp}, we obtain
\begin{equation*}
\dbE^\dbP\Big[f 1_{\{\T<\t\}}+ \ol\cE^\t_L \big[f^{\theta}\big]_{|\theta=(\tau,\B)} 1_{\{\t\le\T\}}\Big] 
=
\mathbb{E}^\mathbb{P} \left[ 
\mathbb{E}^{\nu(\X)}
 \left[ 
f^\theta
 \right]_{|\theta=(\T\wedge\tau,\B)} 
 \right] 
=
\mathbb{E}^{\mathbb{P} \otimes _\tau\nu}
[f]
\leq
 \ol\cE_L [f],
\end{equation*}
which shows  the desired inequality.
Now \eqref{2018-05-30:13}  follows by setting $\tau=\T$ in 
\eqref{2018-05-30:12}.

To prove the last part of the proposition,
 let $\dbP^*\in \cP_L$ be an optimal probability measure such that $\ol\cE_L[f] = \dbE^{\dbP^*}[f]$. 
It follows from
the definition of nonlinear expectation 
and
\eqref{2018-05-30:13} that
\begin{equation*}
  \ol\cE_L[f] = \dbE^{\dbP^*}[f] \le \dbE^{\dbP^*}\Big[ \ol\cE^\T_L \big[f^{\T,\B}\big]\Big] 
\le \ol\cE_L\Big[\ol\cE^\T_L \big[
f^{\theta}\big]_{|\theta=(\T,\B)}\Big] \le \ol\cE_L[f].
\end{equation*}
Therefore, $\dbE^{\dbP^*}[f] = \dbE^{\dbP^*}\Big[ \ol\cE^\T_L \big[f^{\theta}\big]_{|\theta=(\T,\B)}\Big]$. Since we always have $f \leq\ol\cE^\T_L \big[f^{\T,\B}\big]$, we finally obtain $f =  \ol\cE^\T_L \big[f^{\T,\B}\big]$, $\dbP^*$-a.s.
\end{proof}

We conclude this Appendix 
by showing the compactness of the operator $(A-I)^{-1}$ appearing in Section~\ref{sec:hilbert} and providing the proof of 
Proposition~\ref{2018-05-30:08}.

\begin{proposition}\label{2017-02-15:02}
  $(A-I)^{-1}$ is compact.
\end{proposition}
\begin{proof}
  Let $\{x^n\}_{n\in \mathbb{N}}\subset H$ be a sequence
  weakly  convergent 
  to $x$ in $H$.
To prove that $(A-I)^{-1}$ is compact it is sufficient to show
 that $(A-I)^{-1}x^n\rightarrow (A-I)^{-1}x$ strongly in $H$.
Clearly $x^n_0\rightarrow  x_0$ in $\mathbb{R}^m$ and $x^n_1\rightarrow x_1$ weakly in $H'$.
  By looking at \eqref{2017-02-15:00},
  to show that $|(A-I)^{-1}(x^n-x)|_H=|x^n-x|_B\rightarrow 0$, it is then sufficient to show that
  \begin{equation}
    \label{eq:2017-02-15:01}
    \lim_{n\rightarrow \infty}
    \int_{-\infty}^0 \left| \int_t^0 e^{-(s-t)}(x^n_1(s)-x_1(s))ds \right| ^2dt=0.
  \end{equation}
  Since $x^n_1\rightarrow x_1$ weakly, we have
  \begin{equation}
    \label{2017-02-15:03}
\lim_{n\rightarrow \infty}    \int_t^0
e^{-(s-t)} (x^n_1(s)-x_1(s))ds=0\qquad \forall t\in(-\infty,0).
  \end{equation}
By H\"older's inequality and by boundedness of $\{x^n_1\}_{n\in \mathbb{N}}$ in $H'$, we see that 
\begin{equation}
  \label{eq:2017-02-15:05}
 \sup_{\substack{t\in(-\infty,0)\\n\in \mathbb{N}}} \left|\int_t^0e^{-(s-t)} (x^n_1(s)-x_1(s))ds\right|<\infty.
\end{equation}
By \eqref{2017-02-15:03},
\eqref{eq:2017-02-15:05}, and
Lebesgue's dominated convergence theorem, we obtain
\eqref{eq:2017-02-15:01}.
\end{proof}

\begin{proof}[\textbf{Proof of Proposistion~\ref{2018-05-30:08}}]
The fact that
\eqref{2018-03-14:00}
is continuous
 in case $f\in W^{1,2}((0,T),\mathbb{R})$ follows
by
 Lemma~\ref{2018-03-12:00}
and by the discussion in \cite[Remark~2.6]{Rosestolato2017}, or directly  by integration by part  and H\"older's inequality.

We now prove the converse.
We assume $m=1$; the proof for $m>1$ goes along the same lines.
Let $f\in L^2((0,T),\mathbb{R})$ be such that
\eqref{2018-03-14:00} is $d_B$-continuous.
We want to show that $f\in W^{1,2}((0,T),\mathbb{R})$.
Let $V$ denote the subspace of 
$\mathbb{R}\times L^2(\mathbb{R}_-,\mathbb{R})$
whose elements
 are the pairs
$x=(x_0,x_1)$
where $x_1$ is continuous 
on $[-T,0]$ and $x_0=x_1(0)$.
By Lemma~\ref{2018-03-12:00} we have that
the linear functional
\begin{equation*}
  (V,|\cdot|_B)
  \rightarrow
  \mathbb{R},\ (x_0,x_1) \mapsto \int_{-\infty}^0 f(-s)\mathbf{1}_{\{s>-T\}} x_1(s)ds
\end{equation*}
is continuous.
We can then extend it to
a continuous linear functional $\Lambda$
on the
  Hilbert space completion $(H_B,|\cdot|_B)$
of $  ( V,|\cdot|_B)$,
which coincides with the 
completion
of $(H,|\cdot|_B)$ because
$V$ is $|\cdot|_H$-dense in $H$ and $|\cdot|_B$ is weaker than $|\cdot|_H$.

By the Riesz representation theorem it then follows that there exists $z\in H_B$ such that
\begin{equation*}
  \Lambda x= \langle z,x\rangle_B\qquad \forall x\in H_B.
\end{equation*}
By definition of $|\cdot|_B$, 
 $A-I\colon( D(A),|\cdot|_H)\rightarrow (H,|\cdot|_B)$
is an isometry and  then it
can be uniquely extended to
its closure $\overline{A-I}\colon 
(H,|\cdot|_H)\rightarrow (H_B,|\cdot|_B)$
 (we refer to \cite{Rosestolato2017} for further details).
Then we can write 
\begin{equation*}
  \langle z,x\rangle_B=
  \langle (\overline{A-I})^{-1}z,(\overline{A-I})^{-1}x\rangle_H
=
  \langle (\overline{A-I})^{-1}z,(A-I)^{-1}x\rangle_H
=
  \langle 
(A^*-I)^{-1}(\overline{A-I})^{-1}z,x\rangle_H\quad \forall x\in H.
\end{equation*}
Then, we obtain that there exists 
$y\coloneqq (A^*-I)^{-1}(\overline{A-I})^{-1}z$
such that
\begin{equation}\label{2018-03-14:01}
  \langle y,x\rangle_H =\Lambda x\qquad \forall x\in H.
\end{equation}
Since  $y\in  D(A^*)=\mathbb{R}\times W^{1,2}(\mathbb{R}_-)$ and since 
\eqref{2018-03-14:01} holds in particular for all $x\in V$, we conclude $f(-\cdot)\in W^{1,2}((-T,0),\mathbb{R})$.
\end{proof}

\addcontentsline{toc}{section}{References}
\bibliographystyle{plain}
\bibliography{Ren-Rosestolato_FN-PPDE_2018.bll}

\end{document}